\documentclass[11pt]{amsart}
\usepackage{fullpage}
\usepackage{amsmath, amssymb, amsthm}
\usepackage[foot]{amsaddr}

\usepackage{graphicx}
\graphicspath{{./figs_ARM/}} 

\usepackage{color}
\usepackage[colorlinks=true,linkcolor=red,citecolor=blue,urlcolor=cyan]{hyperref}
\newcommand{\eqreftitle}[1]{\texorpdfstring{\eqref{#1}}{Eq.~\ref{#1}}}

\usepackage[ruled,vlined]{algorithm2e}
\usepackage{cleveref}

\newtheorem{thm}{Theorem}[section]
\newtheorem{lem}[thm]{Lemma}
\theoremstyle{remark}
\newtheorem{rem}[thm]{Remark}

\newcommand{\ie}{{\it i.e.}}
\newcommand{\eg}{{\it e.g.}}

\begin{document}
\title{An accelerated rearrangement method  \\ for two-phase composite optimization}\thanks{C.K. acknowledges support from NSF DMS-2208373, NSF DMS-2513176, and Simons Foundation SFI-MPS-TSM-00013519. B.O. acknowledges support from NSF DMS-2136198 and DMS-2513175.}

\author{Chiu-Yen Kao}
\address{Department of Mathematical Sciences, Claremont McKenna College, Claremont, CA 91711 USA}
\email{ckao@cmc.edu}

\author{Seyyed Abbas Mohammadi}
\address{Division of Mathematics, University of Dundee, Dundee DD1 4HN, United Kingdom; School of Computer Science and Applied Mathematics, University of The Witwatersrand, Braamfontein, 2000, Johannesburg, South Africa}
\email{smohammadi001@dundee.ac.uk; abbas.mohammadi@wits.ac.za}  

\author{Braxton Osting}
\address{Department of Mathematics, University of Utah, Salt Lake City, UT 84112 USA}
\email{osting@math.utah.edu}



\begin{abstract} 
We propose and analyze an Accelerated Rearrangement Method (ARM) for solving a class of nonconvex optimization problems involving two-phase composites. These problems include maximizing the  (work) energy of a membrane governed by the Poisson equation and minimizing the principal eigenvalue of a weighted Dirichlet-Laplacian, both subject to material distribution constraints. Building on the classical rearrangement method, we introduce momentum-like acceleration by extrapolating the Fr\'echet derivative, leading to a provably convergent algorithm. We also introduce a restarted variant that guarantees monotonic improvement of the objective.   In one dimension, we derive asymptotic convergence rates for ARM and prove that they improve upon the classical rearrangement method. Numerical experiments in both two and three dimensions confirm the accelerated convergence and demonstrate practical efficiency. 
\end{abstract}

\maketitle

\section{Introduction}
Two-phase composite optimization problems, a class of PDE-constrained optimization problems,
arise in a variety of meaningful real-world applications. Examples include minimizing the fundamental frequency of membranes or plates composed of two materials under area or volume constraints 
\cite{bozorgnia2025optimal,chen2024review,cox1991two,davoli2023spectral,kao2021extremal,krein1955,liu2017converse,mazari2022shape,mohammadi2023maximum,mohammadi2019optimal}; minimizing heat loss in systems made of materials with different thermal conductivities \cite{conca2012minimization,cox1996extremal,kang2020minimization}; designing photonic crystals with tailored optical properties \cite{kao2025semi,kao2005maximizing,Osting:2012,sigmund2008geometric}; optimizing total population through spatial resource allocation \cite{kao2022maximal,mazari2023bang,mazari2022optimisation}; and minimizing the ground state energy in nano-structures \cite{antunes2018nonlinear,mohammadi2016minimization1}, among others. We refer the interested readers to \cite{chambolle2025stability,chanillo2000symmetry,Henrot_2006,lamboley2016properties} and references therein. 

Numerical approaches to these two-phase problems require either explicit parametric representations of the interface between two phases or implicit representations of the interface such as level set  and phase field methods \cite{osher2004level}. The explicit parametric approaches are known for their simplicity and fewer degrees of freedom while implicit approaches allow for flexible topology changes during optimization. Recent developments in level set approaches include two-phase conductivity design problem with nonlinear state dependent diffusion coefficient \cite{allaire2024autofreefem} and two-phase steady-state heat diffusion with an imperfect interface \cite{allaire2021shape}. The phase field methods have recently been  further developed to solve spectral shape and topology optimization \cite{garcke2023phase} and two-phase fluid flows \cite{jain2022accurate} with conservative discrete mass. 

Compared to other numerical approaches for two-phase problems, rearrangement methods (RM) have demonstrated notably fast convergence. Numerous numerical studies have reported that RM typically requires only a few iterations to reach solutions with accuracy comparable to that of the numerical solvers for the governing equations; see, \eg, \cite{chanillo2000symmetry,conca2012minimization,kang2020minimization,kang2017minimization,kao2008principal,kao2013efficient,mohammadi2019optimal}. Despite this excellent numerical performance, relatively few works have rigorously addressed the convergence properties and convergence rates of RM. Recently, a geometric rate of convergence was established for RM applied to a maximization problem involving Poisson equation with Dirichlet boundary conditions, through a detailed analysis of the PDE solution structure and iteration steps \cite{kao2021linear}. Moreover, the convergence of RM—referred to as thresholding algorithms in \cite{chambolle2025stability}—was proved in the context of three optimal control problems: minimization of the Dirichlet energy, optimization of weighted Dirichlet-Laplacian eigenvalues, and a non-energetic control problem with large volume constraints. 
This study extends the analysis of rearrangement-based methods by proposing and investigating an accelerated variant tailored to two-phase composite optimization problems. 

\subsection*{Formulation of the two-phase composite optimization problem.}
 Let $\Omega \subset \mathbb R^d$ be a bounded domain with smooth boundary. Fix constants $f_-, \ f_+, \ \overline{f}$ such that $0 < f_- < \overline{f} < f_+ $. 
Define the \emph{admissible set} $\mathcal{A} = \mathcal{A}(\Omega, f_-, f_+, \overline{f})$ by
$$
\mathcal{A} = \Big\{ f \in L^\infty(\Omega) \colon f = f_- + (f_+ - f_-) \chi_D, \, \, 
D \subset \Omega \, \, \textrm{measurable}, \, \,  
|\Omega|^{-1} \int_\Omega f = \overline{f}  \Big\}. 
$$
Note that $f \in \mathcal{A}$ requires $\overline{f} |\Omega| = \int_\Omega f = f_- |\Omega| + (f_+ - f_-) |D|$, implying $\frac{|D|}{|\Omega|} = \frac{\overline{f} - f_- }{ f_+ - f_-}  =: \delta$, a fixed constant, and so the measure of $D$ is fixed: $|D| = V := \delta |\Omega|$.

We consider a general two-phase composite material optimization problem of the form
\begin{align}
\label{e:gen2phase}
\max_{f \in \mathcal{A}}\:\: J(f) := \mathcal{F}(f, u),
\end{align}
where \( u \) denotes the solution to a PDE on \( \Omega \), with data given by $f$.  The functional 
$J$ is assumed to be convex and Fr\'echet differentiable; precise conditions on 
$J$ will be given in Section~\ref{s:PropARM}. 
Since we are maximizing a convex function and $\mathcal{A}$ is not a convex set, this general PDE-constrained optimization problem  \eqref{e:gen2phase} is generally nonconvex. 
Two  problems of this form are the extremal Poisson problem of maximizing the  (work) energy
\begin{subequations}
\label{e:PoissonOpt}    
\begin{align}
\max_{f \in \mathcal{A}} \ & J(f) = \frac{1}{2} \int_\Omega f(x) u(x) dx  &&  \\ 
\textrm{s.t.} \  & - \Delta u (x) = f(x)  && x\in \Omega \\
& u(x) = 0 && x\in \partial \Omega. 
\end{align}
\end{subequations}
and minimizing the principal eigenvalue of the $f$-weighted Dirichlet-Laplacian,
\begin{subequations}
\label{e:EigOpt}
\begin{align}
\min_{f \in \mathcal{A}} \ & J(f) = \lambda_1(f) \\ 
\textrm{s.t.} \ & - \Delta u(x) = \lambda f(x) u(x)  && x\in\Omega \\
& u(x) = 0 && x\in \partial \Omega.
\end{align}
\end{subequations}
More generally, $\Omega$ could be a Riemannian manifold with boundary and $u$ could be the solution to an elliptic operator (\eg, Laplace–Beltrami operator) on $\Omega$ with data given by $f$.

We can physically interpret \eqref{e:PoissonOpt} and \eqref{e:EigOpt} as follows. 
Consider a clamped membrane $\Omega$ with a two–phase \emph{body-force density} $f\in\{f_-,f_+\}$ and prescribed average $\bar f$ (equivalently, fixed total load $\int_\Omega f=\bar f|\Omega|$).
In \eqref{e:PoissonOpt}, maximizing $J(f)$ places the stronger load $f_+$ to maximize stored elastic energy (work).
In \eqref{e:EigOpt}, $f$ is a \emph{mass density}; minimizing $\lambda_1(f)$ distributes $f_+$ to lower the fundamental frequency under the same mass budget \cite{chanillo2000symmetry,myint2007linear}.

A rearrangement method (RM) for solving \eqref{e:gen2phase} is to start with an initial $f_0 \in \mathcal{A}$ and then for $k\geq 0$, set
\begin{subequations}
\label{e:RM}
\begin{align}
f_{k+1} 
& =
\arg\max_{f \in \mathcal{A}} \,\,  \left\langle J'(f_k), \ f \right\rangle_{L^2(\Omega)} \\
\label{e:RMb}
& = f_- + (f_+ - f_-) \chi_{D_{k+1}},
\qquad \textrm{where} \quad 
D_{k+1} = \left\{ x \in \Omega \colon J'(f_k)(x) \geq c_k \right\} .
\end{align}
\end{subequations}
Here, \( c_k \in \mathbb{R} \) is chosen so that \( |D_{k+1}| = V \).

Since $J$ is convex and Fr\'echet differentiable, we have
\[
J(f_{k+1})-J(f_k)
\;\ge\;
\left\langle J'(f_k),\, f_{k+1}-f_k \right\rangle_{L^2(\Omega)}.
\]
By \eqref{e:RM}, the right-hand side is nonnegative, and therefore
$J(f_{k+1})\ge J(f_k)$.
The formula \eqref{e:RMb} is valid—and the RM is well-defined—only if the level sets of \( J'(f_k) \) have zero measure, ensuring uniqueness (up to null sets) of the superlevel set. While this condition can be readily verified for particular  problems, e.g. \eqref{e:PoissonOpt}--\eqref{e:EigOpt} \cite{kao2021extremal,mohammadi2019optimal}, it is generally difficult to establish and is typically assumed in broader settings.

Under suitable assumptions, the rearrangement method \eqref{e:RM} is convergent \cite{chambolle2025stability} and it is generally observed in numerical experiments that the convergence is linear, \ie, there exists $L \in (0,1)$ such that 
$$
\| f_{k+1} - f^* \|_{L^2(\Omega)} \leq L \| f_k - f^* \|_{L^2(\Omega)} 
\quad \iff \quad 
|D_{k+1} \Delta D^* | \leq L |D_k \Delta D^* |, 
$$
where $f^* = f_- + (f_+ -f_-)\chi_{D^*}$ is the maximizer.
A linear convergence rate has been proven for \eqref{e:PoissonOpt} in one-dimension \cite{kao2021linear}. 

In the convex optimization problem $\min_{x\in \mathbb R^n} \, \eta(x)$ with $\eta\colon \mathbb R^n \to \mathbb R$ a convex function, it is well-known that the Nesterov's accelerated gradient method (AGM) yields improves convergence rates over gradient descent methods   \cite{nesterov1983method,RyuYin2023,su2016differential}. AGM starts with an initial point $y_0 \in \mathbb R^n$ and for $k\geq 0$ takes 
\begin{subequations}
\label{e:Nesterov}
\begin{align}
\label{e:NesterovA}
x_{k+1} &= y_k - s \nabla \eta(y_k), \\
\label{e:NesterovB}
y_{k} &= x_{k} + \theta_k \left( x_k - x_{k-1}\right),
\end{align}
\end{subequations}
where  \( s > 0 \) is a step size parameter and $\theta_k = \frac{k}{k + 3} \in [0,1)$ is an extrapolation parameter satisfying $\theta_k = 1 - O(1/k)$. 
Iterates $\{x_k\}$ of AGM applied to a convex, $L$-smooth function $\eta$ converge to a minimizer $x^*$ with rate 
$\eta(x_k) - \eta(x^*) \leq O \left(   \frac{ \|x_0 - x^*\|^2}{s k^2} \right)$ \cite{RyuYin2023,su2016differential}. That is, AGM has sublinear convergence and is an improvement over the linear $O(1/k)$ convergence rate for gradient descent.

\emph{In this paper, inspired by the speedup seen in the AGM \eqref{e:Nesterov}, we seek to find an accelerated rearrangement method (ARM).} 
Nesterov's acceleration cannot straightforwardly be applied for several reasons:
(i) Nesterov-type acceleration methods have been generalized for convex  PDE-constrained optimization problems (see, \eg, \cite{clason2017primal,nesterov2018lectures});  our PDE-constrained optimization problem is non-convex. 
(ii) Our optimization problem \eqref{e:gen2phase} has constraints, so iterates of an extrapolation step  of the form in \eqref{e:NesterovB}, \ie, $f_{k} + \theta_k ( f_k - f_{k-1})$ may not be in $\mathcal{A}$. 
In particular, the updated $f_k$ might be negative on a subset of $\Omega$ and hence the PDE and objective functional
may no longer be well-defined in the intended sense.
Because the problem is not convex and constrained, variants of AGM such as FISTA \cite{BeckTeboulle2009} or iPIANO \cite{OchsChenBroxPock2014} do not directly apply.

We propose the following \emph{accelerated rearrangement method (ARM)}, where we start with an initial $f_0 \in \mathcal{A}$ and for $k\geq 0$, set 
\begin{subequations}
\label{e:ARM}
\begin{align}
\label{e:ARMa}
f_{k+1} &= \arg\max_{f \in \mathcal{A}} \,\,  \left\langle g_k, \ f \right\rangle_{L^2(\Omega)}, 
\quad \textrm{where}  \quad 
g_k = J'(f_k)
\;+\; \theta_{k}\left( J'(f_k) - J'(f_{k-1})\right)  \\
&= f_- + (f_+ - f_-) \chi_{D_{k+1}}, 
\quad \textrm{where} \quad 
D_{k+1} = \left\{ x \in \Omega \colon g_k(x) \geq c_k \right\} .
\end{align}
\end{subequations}
Here, \( c_k \in \mathbb{R} \) is again chosen so that \( |D_{k+1}| = V \),
and $\theta_k = \frac{k}{k+3} \in [0,1)$  so that we accelerate by ``extrapolating'' on the current and previous Fr\'echet derivatives yielding a function, $g_k$, used for thresholding. The updated iterate $f_{k+1}$ (maximizer in \eqref{e:ARMa}) is a uniquely defined (up to null sets) \emph{admissible} function provided that the level sets of $g_k$ have zero measure—a condition that ensures well-posedness of the thresholding step and is typically assumed in practice. If the level sets of \( g_k \) have positive measure, the algorithm can still be effectively used; see Remark~\ref{rem:Nonuniqness_f_k+1} for further discussion.

We also consider a \emph{restarted accelerated rearrangement method (RARM)}, where, for $ k \geq 0$,  we define 
\begin{subequations}
\label{e:RARM}
\begin{align}
\tilde{f}_{k+1} &= \arg\max_{f \in \mathcal{A}} \,\,  \left\langle g_k, \ f \right\rangle_{L^2(\Omega)}, 
\quad \textrm{where}  \quad 
g_k = J'(f_k)
\;+\; \theta_{k}'\left( J'(f_k) - J'(f_{k-1})\right), \\
&= f_- + (f_+ - f_-) \chi_{D_{k+1}}, 
\quad \textrm{where} \quad 
D_{k+1} = \left\{ x \in \Omega \colon g_k(x) \geq c_k \right\}. 
\end{align}
\end{subequations}
Here, $\theta_k' = \frac{k - k_0}{k - k_0 +3}$ with $k_0 = 0$, initially. 
Then, if $J(\tilde{f}_{k+1}) > J(f_{k})$, 
we take $f_{k+1} = \tilde{f}_{k+1}$. Otherwise,  we update $k_0 = k$ in $\theta_k'$, so that $\theta_k' = 0$ and take the RM update 
$f_{k+1} = \arg\max_{f \in \mathcal{A}} \,\,  \left\langle J'(f_k), \ f \right\rangle_{L^2(\Omega)}$. 
RARM has the property that the objective value at each iterate is non-decreasing, as the RM method is non-decreasing.

\subsection*{Summary of results and outline.}
In \cref{s:PropARM}, we establish some theoretical properties of the accelerated rearrangement method (ARM) \eqref{e:ARM}.
In \cref{s:PropARMa}, we first prove that RARM has the property that, under mild assumptions, every weak-$*$ accumulation point of its iterates lies in the admissible class and satisfies the first‐order necessary condition for optimality; see \cref{t:Convergence}.
In \cref{s:PropARMc} we  introduce a modified ARM for the simple one-dimensional extremal Poisson problem \eqref{e:PoissonOpt}. 
Write the iterates of the modified ARM as 
$f_k(x)=f_{-} + (f_{+}-f_{-})\,\chi_{[\,y_k-\delta,\;y_k+\delta\,]}(x)$.
Linearizing the 1D threshold‐map, $h$, about its fixed point reveals a contractive factor $L<1$, and by optimally choosing the extrapolation weight $\theta^*$ one can collapse the characteristic equation onto a double root $r^*<L$.  Theorem \ref{t:ConvRate1d} then proves that the modified ARM attains an improved asymptotic rate
\[
|y_{k+1}|\;\lesssim\;r^*\,|y_k|,
\qquad \qquad 
r^*=1-\sqrt{1-L}<L,
\]
so that the width‐parameter $y_k\to0$ and hence $f_k \to f^*$ at a faster geometric rate as compared to (unaccelerated) RM.  
Numerical experiments confirm these rate gains in both low- and high-contrast regimes.

\begin{figure}[p!]
\begin{center}
\includegraphics[width=0.45\textwidth]{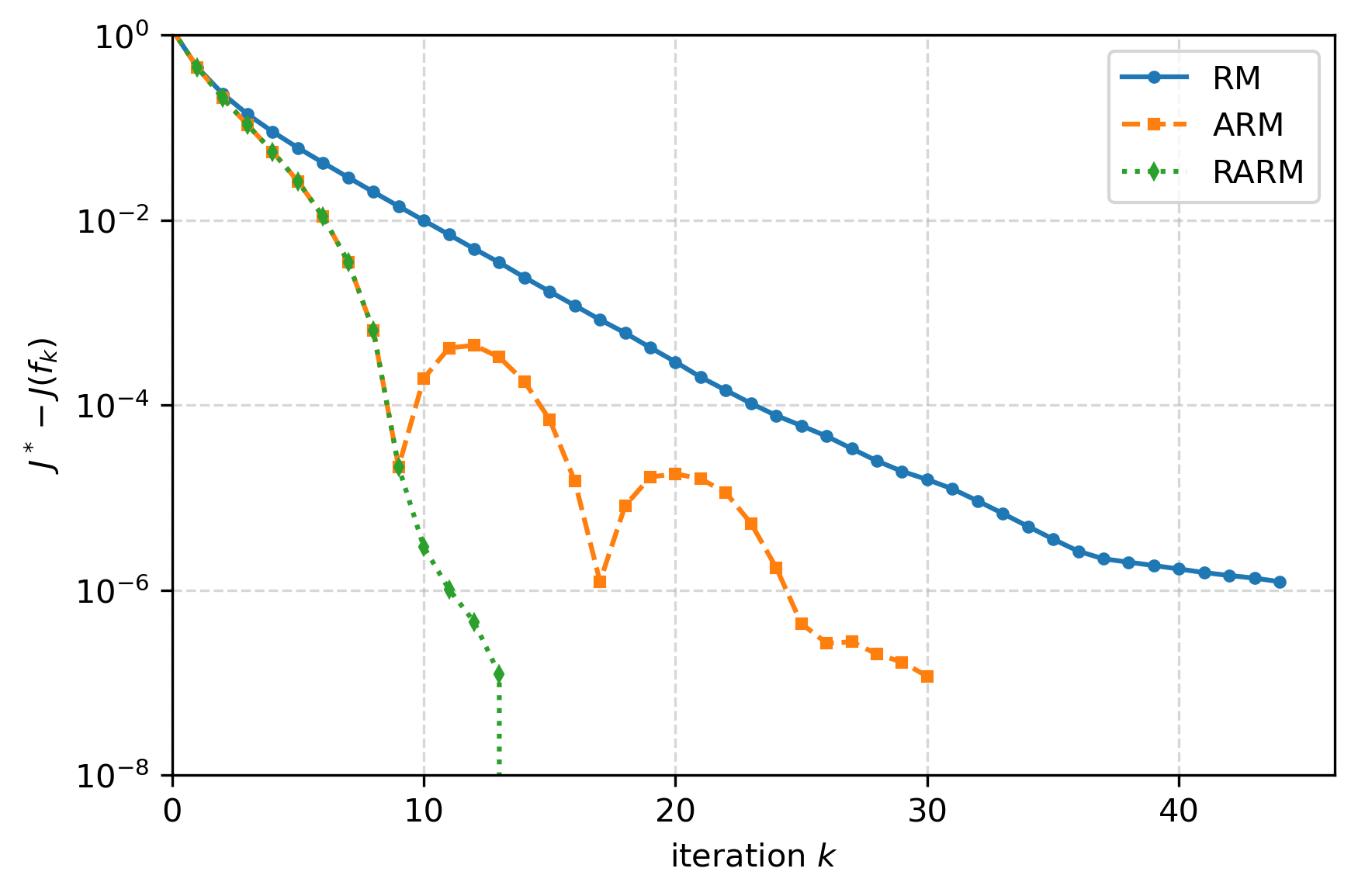}
\includegraphics[width=0.45\textwidth]{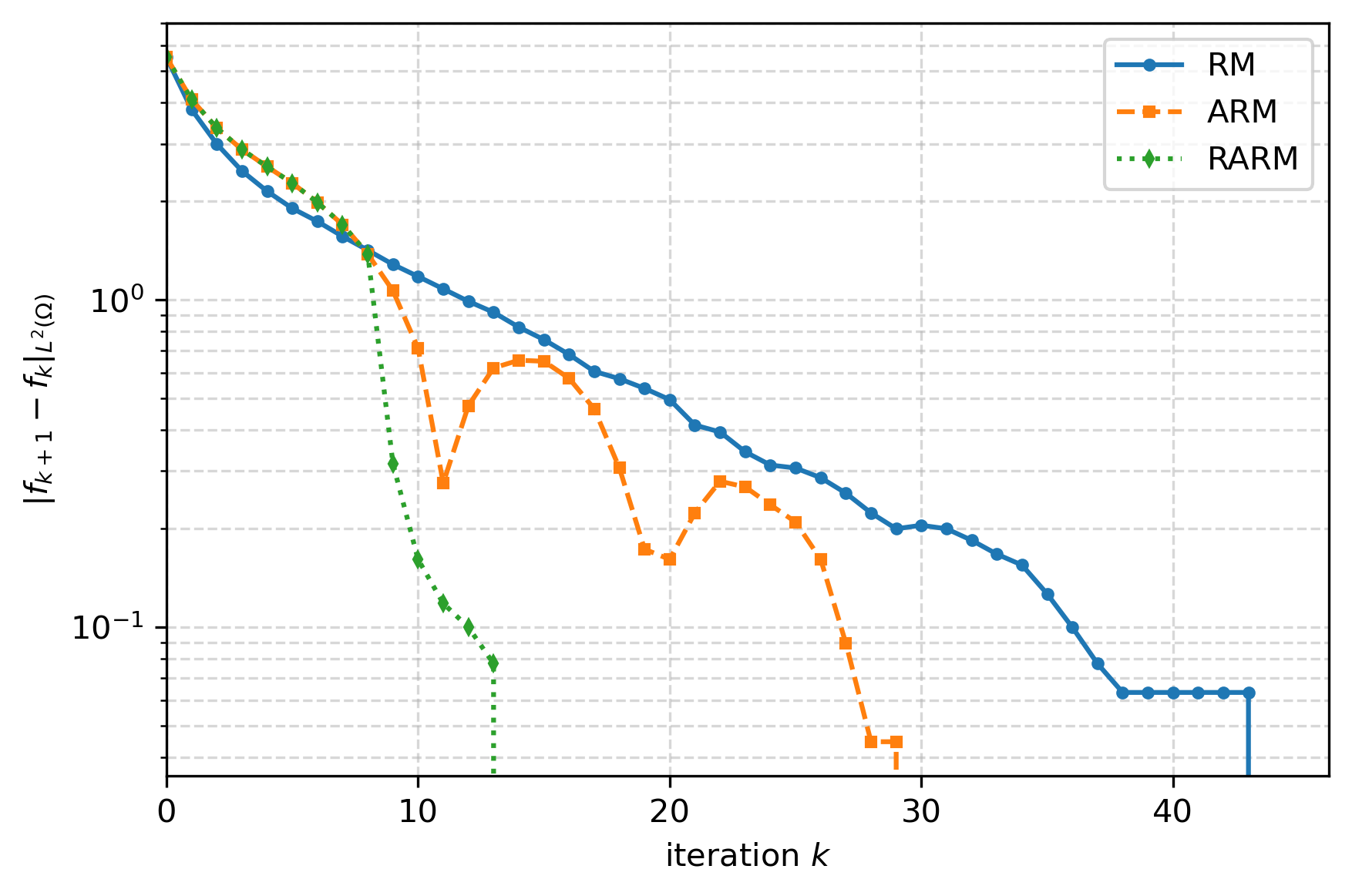} \\ 
\medskip
\includegraphics[width=0.45\textwidth,trim={7cm 0 0 4cm},clip]{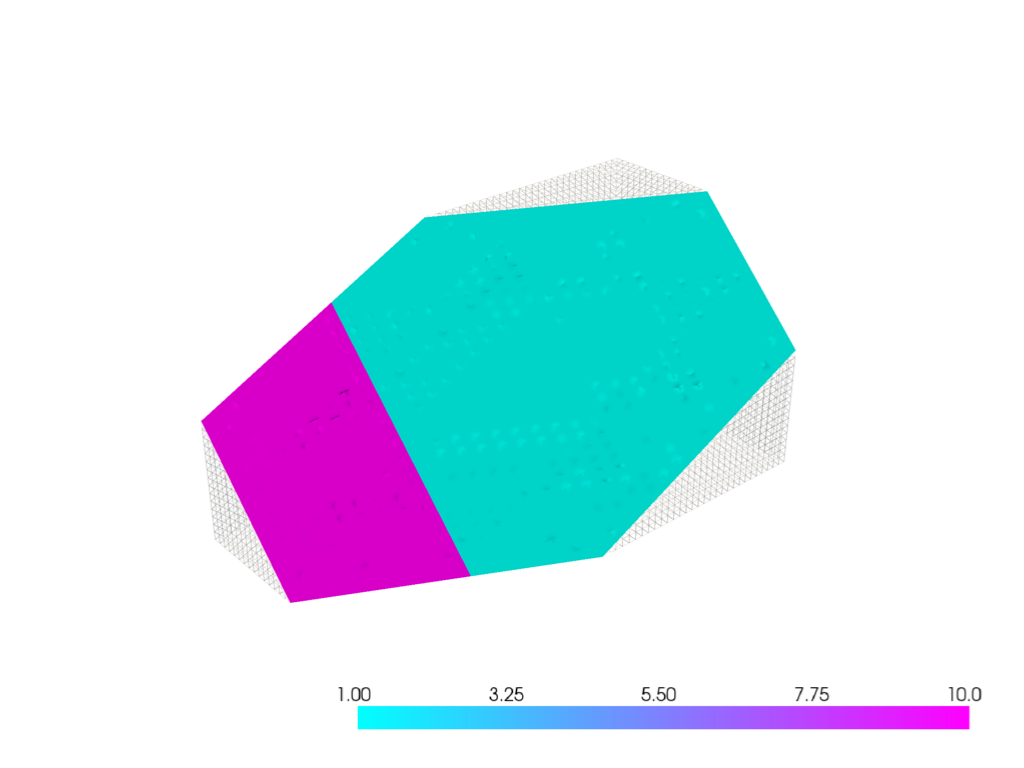}
\includegraphics[width=0.45\textwidth,trim={7cm 0 0 4cm},clip]{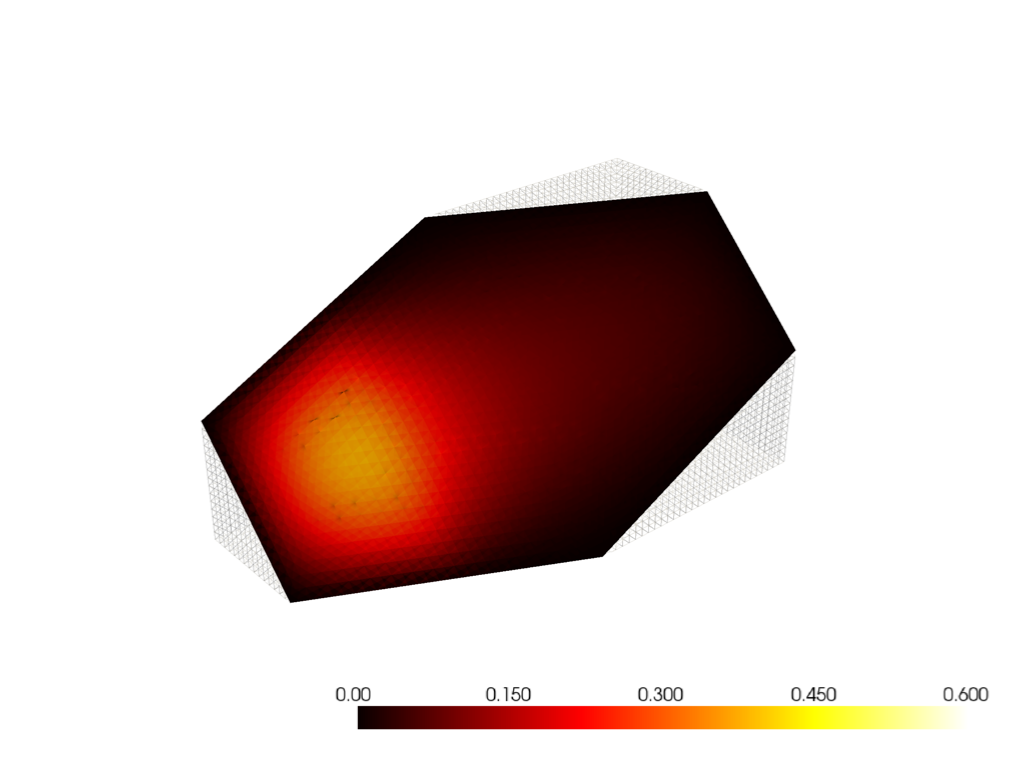} \\ 
\medskip
\includegraphics[width=0.45\textwidth,trim={7cm 0 0 4cm},clip]{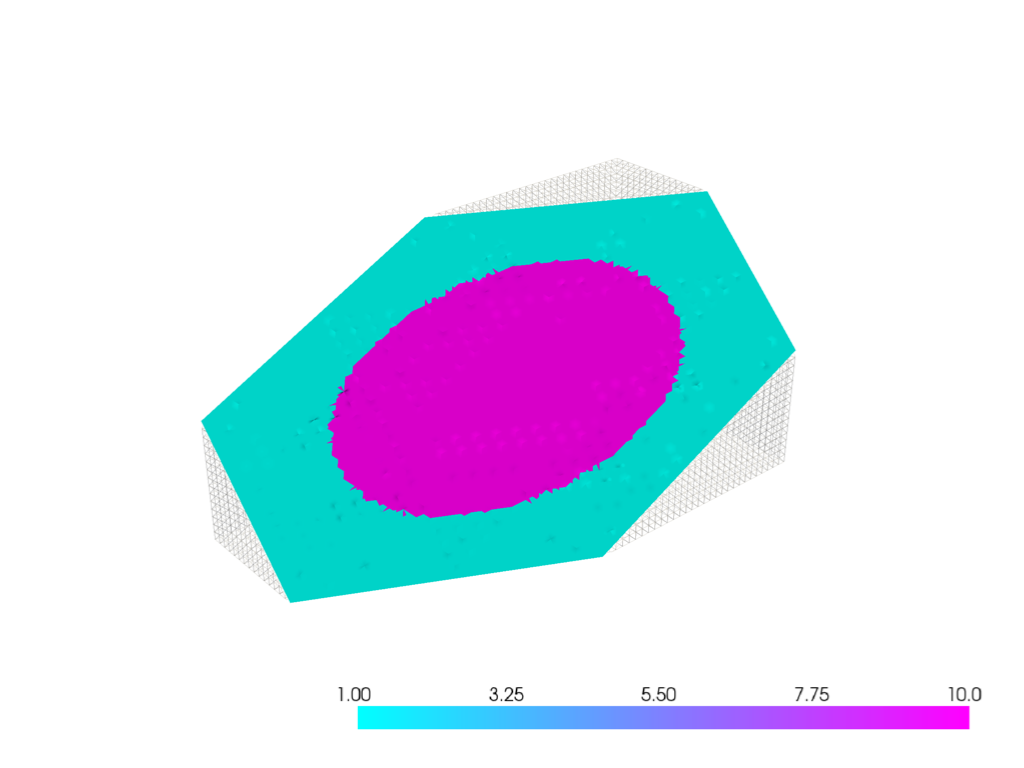}
\includegraphics[width=0.45\textwidth,trim={7cm 0 0 4cm},clip]{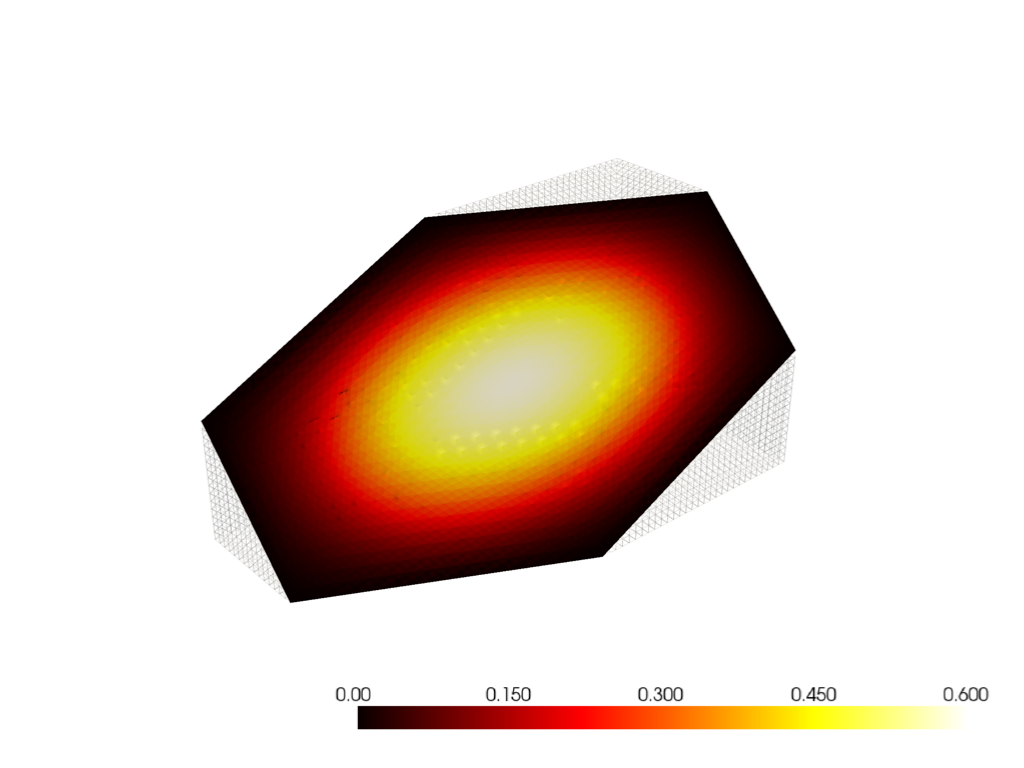} 
\end{center}
\caption{We consider the extremal Poisson  problem \eqref{e:PoissonOpt} on the cuboid $\Omega = [0,2]\times[0,1]\times [0,1]$, with $f_- = 1$, $f_+ = 10$, and $\delta = 0.28$. 
We plot the iteration $k$ vs. $J^* - J(f_k)$ {\bf (top left)} and $\|f_{k+1} - f_{k}\|_{L^2(\Omega)}${\bf (top right)} for the 
rearrangement method (RM) \eqref{e:RM},  
accelerated RM (ARM)\eqref{e:ARM}, and 
restarted ARM (RARM) \eqref{e:RARM}.  
In the middle panels, we plot the initial density $f_0$ {\bf (middle left)} and initial solution  $u_0$ {\bf (middle right)}. The density and solution are plotted on a ``diagonal slice'' though the cuboid mesh. 
In the lower panels, we plot the final density $f^*$ {\bf (bottom left)} and final solution  $u^*$ {\bf (bottom right)}. 
See \cref{s:CompCuboid} for details. }
    \label{fig:Cuboid}
\end{figure}

In \cref{s:CompRes}, we present results of several computational experiments demonstrating and comparing the RM \eqref{e:RM}, ARM \eqref{e:ARM}, and RARM \eqref{e:RARM} on the extremal Poisson problem \eqref{e:PoissonOpt} and extremal eigenvalue problem \eqref{e:EigOpt}. 
Generally, we observe that the RM converges monotonically, but at a slower rate than the ARM. The ARM exhibits the classical ``bouncing'' behavior of an accelerated method. The RARM converges monotonically and because the ``bouncing'' is eliminated, converges in fewer iterations than the ARM. 
For a preview of the typical performance of these algorithms, we can consider the extremal Poisson problem \eqref{e:PoissonOpt} on the three-dimensional cuboid $\Omega = [0,2]\times [0,1] \times [0,1]$, as further discussed in \cref{s:CompCuboid}. The results are displayed in \cref{fig:Cuboid}.
Indeed, in this example, the accelerated methods converge in fewer iterations than the original rearrangement method.

In summary, our theoretical and numerical results show that ARM converges geometrically (exponentially) with a better rate than the geometrically convergent RM. We remark that this speedup is in contrast with that for the Nesterov's accelerated gradient method (AGM) \eqref{e:Nesterov}, which gives a sublinear $O(1/k^2)$ convergence rate, as compared to linear $O(1/k)$ for the gradient method. We conclude in \cref{s:Disc} with further discussion.


\section{Properties of the accelerated rearrangement method}
\label{s:PropARM}
We first briefly discuss the existence of a solution for \eqref{e:gen2phase}.
Define the set
$$
\mathcal{B} :=  \overline{\textrm{co}}(\mathcal A) =\Big\{ f\in L^\infty(\Omega)\colon f_- \le f(x)\le f_+, \ \tfrac{1}{|\Omega|}\int_\Omega f=\overline{f} \Big\}.
$$
Throughout this section, we assume the objective functional  $J\colon \mathcal B \subset L^2(\Omega)\to\mathbb R$ is convex, Fr\'echet differentiable, and
\emph{sequentially weak-$*$ continuous on $\mathcal B$}, that is,
whenever $f_n \overset{*}{\rightharpoonup} f$ in $L^\infty(\Omega)$ with $f_n,f\in\mathcal B$,
we have $J(f_n)\to J(f)$. 
It is well-known that $\mathcal{B}$ is convex, compact for the weak-$*$ convergence and
its extremal points are exactly the bang-bang functions of the form  
$f = f_- + (f_+ - f_-) \chi_D$, $|D| = V$,
 \ie, all admissible functions in $\mathcal{A}$  \cite{friedland1977extremal,Henrot_2006}.
Considering the maximization over the relaxed admissible set $\mathcal{B}$, \ie, 
$\max_{f \in \mathcal{B}} \, J(f)$,
the weak-$*$ compactness of $\mathcal{B}$   and weak-$*$ continuity of $J$  ensure the existence of a maximizer. 
Furthermore, since $J$ is convex in $f$, its maximizer over the convex set $\mathcal{B}$ 
must lie at an extreme point of $\mathcal{B}$, which is necessarily a bang–bang function in $\mathcal{A}$. 
Hence, problem~\eqref{e:gen2phase} admits a solution.

\subsection{Stationarity of accumulation points under restarted updates}
\label{s:PropARMa}

In this subsection, we study the stationarity properties of RARM and show that weak-$*$ accumulation points generated at restart iterations satisfy the first-order necessary optimality condition.


\begin{thm}[Stationarity of accumulation points for RARM]
\label{t:Convergence}
In \eqref{e:gen2phase}, assume that $J\colon \mathcal B \subset L^2(\Omega)\to\mathbb R$ is convex, Fr\'echet differentiable, and
\emph{sequentially weak-$*$ continuous on $\mathcal B$}.
In addition, suppose that
\[
f_n \overset{*}{\rightharpoonup} f \text{ in } L^\infty(\Omega)
\quad \Longrightarrow \quad
J'(f_n)\to J'(f) \text{ in } L^1(\Omega).
\]
Consider any sequence $\{f_k\}\subset\mathcal A$ generated by RARM \eqref{e:RARM} such that
there exists an infinite set of indices $\mathcal K\subset\mathbb N$ with the property that
for every $l\in\mathcal K$, $\theta^\prime_{l}=0$, i.e., $f_{l+1}$ is determined using
the restarted (RM) update. Then:\\
\smallskip
 \noindent
(i) Every weak-$*$ accumulation point $\hat f$ of the subsequence $\{f_{l}\}_{l\in\mathcal K}$
belongs to $\mathcal B$ and satisfies the first-order necessary optimality condition
\begin{equation*}
\langle J'(\hat f),\, f-\hat f\rangle \le 0 \qquad \forall\, f\in\mathcal B.
\end{equation*}
\smallskip
\noindent
(ii) If the level sets of $\hat g:=J'(\hat f)$ have zero Lebesgue measure, then
$\hat f\in\mathcal A$ and $\hat f=f_-+(f_+-f_-)\chi_{\hat D}$ with
\[
\hat D=\{x\in\Omega\colon \hat g(x)\ge c\},\qquad
c=\sup\big\{s\in\mathbb R\colon \ |\{x\colon \hat g(x)\ge s\}|\ge V\big\}.
\]
\end{thm}

\begin{proof}
(i) Let $\hat f$ be a weak-$*$ accumulation point of $\{f_{l}\}_{l\in\mathcal K}$.
Passing to a subsequence (still indexed by $l\in\mathcal K$), we may assume
$f_{l}\overset{*}{\rightharpoonup}\hat f$ in $L^\infty(\Omega)$.
Since $f_{l}\in\mathcal A\subset\mathcal B$ and $\mathcal B$ is weak-$*$ compact,
we have $\hat f\in\mathcal B$. This implies that we have
$J'(f_l)\to J'(\hat f)$ in $L^1(\Omega)$ in view of the assumptions of the theorem.

Fix $l\in\mathcal K$. By \eqref{e:RARM},
\[
\int_\Omega f_{l+1}\,J'(f_l)\,dx \ge \int_\Omega f\,J'(f_l)\,dx
\qquad \forall\, f\in\mathcal A.
\]
In view of the compactness of $\mathcal B$, there is $\bar f \in \mathcal{B}$ such that up to a subsequence
$f_{l+1} \overset{*}{\rightharpoonup}\bar f$.
Passing to the limit in the inequality above and using the convergences above and the compactness of $\mathcal{B}$, we obtain
\begin{equation}\label{ineq:barf}
    \int_\Omega \bar f \, J'(\hat f) \, dx 
\, \geq \, 
\int_\Omega f \, J'(\hat f) \, dx, \quad \text{for all } f \in \mathcal{B}.
\end{equation}
We claim that  $\int_\Omega \bar f \, J'(\hat f) \, dx 
\, = \, 
\int_\Omega  \hat f \, J'(\hat f) \, dx$ and so \eqref{ineq:barf} yields 
\[
\langle J'(\hat f),\, f-\hat f\rangle \le 0 \qquad \forall\, f\in\mathcal B.
\]

To prove the claim, we assume that 
$\int_\Omega \bar f \, J'(\hat f) \, dx 
\, > \, 
\int_\Omega  \hat f \, J'(\hat f) \, dx$. Since $J$ is convex and Fr\'echet differentiable, we have

\begin{equation}\label{ineq:J(bf)>J(h f)}
 J(\bar f) -J(\hat f) \geq    \langle J'(\hat f),\, \bar f-\hat f\rangle >0.
\end{equation}

  By construction of RARM, the objective values satisfy $J(f_{k+1}) \geq J(f_k)$  for all $k$.
The sequence $\{J(f_k)\}$  is also bounded from above due to existence of a maximizer for \eqref{e:gen2phase}. Therefore, we have $J(f_k) \to J^*$. Moreover,  sequences $J(f_l)$ and $J(f_{l+1})$ should also converge to $J^*$ due to weak-$*$ continuity of $J$. This means $J(\hat f) = J(\bar f) = J^*$ which contradicts  \eqref{ineq:J(bf)>J(h f)}.

(ii) From the proof of (i), \( \hat{f} \) is the maximizer of the linear functional
$L(f):=\int_{\Omega} f \, \hat g\, dx$,
over all \( f \in \mathcal{B} \). {In fact}, \( \hat{f} \) is the unique maximizer up to null sets, since the level sets of \( \hat g \) have measure zero \cite[Lemma 2.4]{burton1989variational}.
Therefore, the final assertion follows from the bathtub principle; see \cite{lieb2001analysis} and \cite[Lemma 2.2]{mohammadi2023maximum}.
\end{proof}

\begin{rem}\label{rem:Nonuniqness_f_k+1}
In ARM or RARM, one solves the maximization problem \eqref{e:ARMa} using thresholding. In general, a maximizer of \( \int_\Omega f g_k \) over \( \mathcal{A} \) is given by \( f_{k+1} = f_- + (f_+ - f_-)\chi_{D_{k+1}} \), where 
\begin{equation}\label{e:dk+1asasubset}
 \{x \colon  g_k(x) > c_k\} \subset D_{k+1} \subset \{x\colon g_k(x) \geq c_k\},\:\:c_k = \sup\{c \in \mathbb{R}\colon|\{x \in \Omega \colon g_k(x) \geq c\}| \geq V\},   
\end{equation}
see, \eg, \cite{lieb2001analysis} and \cite[Lemma 5]{kao2022maximal}. It is clear that if the level sets of \( g_k \) have measure zero, then \( D_{k+1} \) is uniquely determined, \ie, \( D_{k+1} = \{x \in \Omega \colon g_k(x) \geq c_k\} \).
Showing that the level sets of \( g_k \) have measure zero is generally nontrivial, but flat regions are typically not observed in numerical practice due to discretization. This justifies assuming the condition in practice. Note that even if the level set \( \{x \in \Omega \colon g_k(x) = c_k\} \) has positive measure—resulting in non-uniqueness of \( D_{k+1} \)—the algorithm still functions correctly. In this case, any admissible choice of \( D_{k+1} \) satisfying \eqref{e:dk+1asasubset} can be used. The conclusions of Theorem~\ref{t:Convergence}  for RARM remain valid under this relaxed condition.
\end{rem}

\begin{rem}
   If $\{f_k\}$ is a sequence generated by ARM such that 
\( f_k \overset{*}{\rightharpoonup} \hat{f} \) in \( L^\infty(\Omega) \), and $\hat{f} \in \mathcal{A}$, 
then $f_k \to \hat{f}$ in $L^2(\Omega)$ by the Radon–Riesz theorem, since all functions in $\mathcal{A}$ 
have the same $L^2$-norm.
\end{rem}
\begin{rem}
\label{rem:RARM_restart}
The assumption that the restarted update (i.e., $\theta'_l=0$) occurs infinitely often
is mild and consistent with practical performance of the algorithm.
Indeed, in our numerical experiments, we observe that the extrapolated step in RARM is accepted for a small number of consecutive iterations, after which a restart is triggered to recover monotonicity of the objective.
Assuming that this behavior persists, the assumption required in Theorem~\ref{t:Convergence} holds. 
\end{rem}

To conclude this subsection, we show that our example problems \eqref{e:PoissonOpt}–\eqref{e:EigOpt} satisfy the assumptions of Theorem \ref{t:Convergence}. The objective functionals in both the extremal Poisson problem \eqref{e:PoissonOpt} and the eigenvalue problem \eqref{e:EigOpt} are  sequentially weak-$*$ continuous and
Fr\'echet differentiable; see, \eg, \cite{chambolle2025stability,Henrot_2006,troltzsch2010optimal}.
It is not difficult to compute that 
$J'(f_k) = u_k$ 
in \eqref{e:PoissonOpt} and 
$J'(f_k) = - \lambda u_k^2$ 
in \eqref{e:EigOpt}
(assuming that the eigenfunction has been normalized so that $\int_\Omega f_k u_k^2 = 1$).
It is easy to check that the objective functional in \eqref{e:PoissonOpt} is convex. Moreover, $1/\lambda_1(f)$ is convex in \eqref{e:EigOpt}, \cite[Theorem 9.1.3]{Henrot_2006}.
We  show that  for any sequence $\{f_k\}\subset \mathcal{A}$, such that \( f_k \overset{*}{\rightharpoonup} \hat{f} \) in \( L^\infty(\Omega) \), we have
\[
J'(f_k) \to J'(\hat f) \quad \text{in } L^1(\Omega).
\]
We provide a detailed proof for the eigenvalue problem \eqref{e:EigOpt}; the corresponding result for \eqref{e:PoissonOpt} follows by a similar argument and is omitted for brevity.

\begin{lem}\label{l: frechetdiffeigopt}
Consider the functional \( J(f) \) corresponding to problem \eqref{e:EigOpt}. If \( \{f_k\} \subset \mathcal{A} \) and \( f_k \overset{*}{\rightharpoonup} \hat{f} \) in \( L^\infty(\Omega) \), then
$J'(f_k) \to J'(\hat f)$ in $L^1(\Omega)$.

\end{lem}
\begin{proof}
We recall that  a function \( u \in H^1_0(\Omega) \setminus \{0\}\), is an eigenfunction for the PDE in \eqref{e:EigOpt}, corresponding to the principal eigenvalue \( \lambda_1(f) >0 \), if it satisfies
\begin{equation}\label{e:eigpdevarform}
    \int_\Omega \nabla u \cdot \nabla \phi = \lambda_1(f) \int_\Omega f u \phi, \quad \text{for all } \phi \in H^1_0(\Omega),
\end{equation}
and that the eigenvalue \( \lambda_1(f) \) admits the variational characterization
\begin{equation}\label{e:lambda_1varform}
    \lambda_1(f) = \underset{u \in H^1_0(\Omega),\, u \neq 0}{\min} \frac{\int_\Omega |\nabla u|^2}{\int_\Omega f u^2}.
\end{equation}
For clarity, we denote the eigenfunction corresponding to a given \( f \) by \( u_f := u \), to emphasize its dependence on \( f \).

For any $f\in \mathcal{A}$ and $u\in H_0^1(\Omega)\setminus\{0\}$, we have
\[
\frac{\int_\Omega |\nabla u|^2\,dx}{\int_\Omega f\,u^2\,dx}
\le
\frac{\int_\Omega |\nabla u|^2\,dx}{\int_\Omega f_-\,u^2\,dx}.
\]
Taking the minimum over $u\neq 0$ and using \eqref{e:lambda_1varform} yields
\(
\lambda_1(f)\le \lambda_1(f_-).
\)
 Hence, the sequence \( \{ \lambda_1(f_k) \} \) is bounded from above and therefore admits a convergent subsequence (which we denote, for simplicity, by the same index) such that \( \lambda_1(f_k) \to \lambda \). Let us assume that all corresponding eigenfunctions are normalized so that \( \int_\Omega f_k u_{f_k}^2 = 1 \). Then, using \eqref{e:eigpdevarform}, we obtain
\[
\int_\Omega |\nabla u_{f_k}|^2 = \lambda_1(f_k) \int_\Omega f_k u_{f_k}^2 = \lambda_1(f_k)  \leq \lambda_1(f_-),
\]
which shows that the sequence \( \{ u_{f_k} \} \) is bounded in \( H^1_0(\Omega) \). Consequently, there exists a subsequence (still denoted with the same index) such that
\begin{equation}\label{e:u_f-k_conv}
   \nabla u_{f_k} \rightharpoonup \nabla u, \quad u_{f_k} \to u \quad \text{in } L^2(\Omega).
\end{equation}
In view of \eqref{e:eigpdevarform}, we have
\[
\int_\Omega \nabla u_{f_k} \cdot \nabla \phi = \lambda_1(f_k) \int_\Omega f_k u_{f_k} \phi, \quad \text{for all } \phi \in H^1_0(\Omega),
\]
and by passing to the limit using \eqref{e:u_f-k_conv}, we obtain
\[
\int_\Omega \nabla u \cdot \nabla \phi = \lambda \int_\Omega \hat{f} \, u \phi, \quad \text{for all } \phi \in H^1_0(\Omega).
\]
In addition, \eqref{e:u_f-k_conv}  and the $ f_k \overset{*}{\rightharpoonup} \hat{f}$ imply that
\[ \int_\Omega \hat f u_{{\hat f}}^2 = \lim_{k \to \infty} \int_\Omega f_k u_{f_k}^2 =1. \]
Hence, \( u \) is nontrivial
which shows that \( u \) is an eigenfunction corresponding to \( \hat{f} \).

It is well-known that the principal eigenfunction of the PDE in \eqref{e:EigOpt} does not change sign in \( \Omega \), and it is the only eigenfunction with this property; see, \eg, \cite{Henrot_2006}, \cite[Theorem 5.1]{le2006eigenvalue}. Therefore, we may assume that all functions \( u_{f_k} \) are chosen to be positive in \( \Omega \), and by \eqref{e:u_f-k_conv}, we also have a.e. pointwise convergence \( u_{f_k}(x) \to u(x) \) which means the eigenfunction $u(x)$ is positive and so it is the first eigenfunction.
 In summary, we conclude that \( \lambda = \lambda_1(\hat{f}) \) and \( u = u_{\hat{f}} \), and hence
\[
\lambda_1(f_k) u^2_{f_k} \to \lambda_1(\hat{f}) u^2_{\hat{f}} \quad \text{in } L^1(\Omega),
\]
which completes the proof.
\end{proof}
\begin{rem}
 In the Poisson case \eqref{e:PoissonOpt}, \( J'(f) = u \), where \( u \) solves the boundary value problem \eqref{e:PoissonOpt}; in the eigenvalue case \eqref{e:EigOpt}, \( J'(f) = -\lambda u^2 \), where \( u \) is the normalized principal eigenfunction. In both cases, it follows from \cite[Lemma~7.7]{GilbargTrudinger2001} that the level sets of \( J'(f) \) have zero Lebesgue measure. Therefore, the assumption in \cref{t:Convergence}(ii) is satisfied for these applications.
\end{rem}

\subsection{Convergence rate for a modified ARM for the 1D Poisson problem}
 \label{s:PropARMc}
For the ARM \eqref{e:ARM} applied to the 1D extremal Poisson problem \eqref{e:PoissonOpt}, we accelerate by extrapolating $J^\prime(f_k) =  u_k$. 
Here, we consider a modification of this algorithm, which is similar in spirit, where we can analyze to see an improved convergence rate. 

Following the notation in \cite{kao2021linear}, we consider the one‐dimensional rearrangement map
\[
y_{k+1} \;=\;h(y_k),
\]
where \(y_k\in[0,1]\) is the center of the support of the width $2\delta$ bang–bang density, 
\[
f_k(x)=f_{-} + (f_{+}-f_{-})\,\chi_{[\,y_k-\delta,\;y_k+\delta\,]}(x).
\]
Here, the map $h\colon [0,1] \to [0,1]$ is defined
\[
h(y) :=\begin{cases}
\frac{\delta}{\alpha}(f_+-f_-)(1-y) & \text{if } y>\lambda \\
y-\frac{2}{f_+-f_-}\left[f_+\delta-\sqrt{f_+^{2}\delta^{2}-y\delta(f_+-f_-)\left(f_-+\delta(f_+-f_-)\right)}\right] & \text{otherwise.}
\end{cases}
\]
and $\lambda:=\max\{3\delta-1,\frac{\delta(f_-+f_+)}{f_-+\delta(f_+-f_-)}\}$. It is shown in \cite[Theorem 1]{kao2021linear} that near the fixed point \(y=0\),
\[
h(y)\approx L\,y,
\qquad
L=\left(1-\frac{f_{-}}{f_{+}}\right)\,(1-\delta)\;<\;1.
\]
Consequently, one obtains the convergence rate for the RM,  
$$
|y_{k+1}|\;\le\; L \,|y_k|.
$$

For fixed $\theta \in (0,1]$,  we consider a \emph{modified ARM} based on accelerating the update on  \(y\):
\begin{equation}
\label{e:modifiedARM}    
\tilde y_k \;=\;y_k \;+\;\theta_k \,(y_k - y_{k-1}),
\qquad
y_{k+1} \;=\;h(\tilde y_k).
\end{equation}
For a fixed optimal choice $\theta_k \equiv \theta^* \in (0,1]$, the following theorem gives an improved rate for this \emph{optimal ARM method}. 

\begin{thm}[Geometric asymptotic convergence rate for optimal ARM \eqref{e:modifiedARM}]
\label{t:ConvRate1d}
With \(\theta_k=\theta^* = \frac{2 - L \;-\;2\sqrt{1-L}}{L}\), the modified ARM \eqref{e:modifiedARM} satisfies, for  large $k$,
\[
|y_{k+1}|\;\lesssim \;r^*\,|y_k|,
\qquad
r^* = 1-\sqrt{1-L}\;<\;L
\]
where $L<1$.
Hence
$|y_k| = O \left( (r^*)^k \right)$,
and since
\(\|f_k-f^*\|_{L^2}=(f_{+}-f_{-})\sqrt{2\,|y_k|}\),
one obtains
\[
\|f_k-f^*\|_{L^2}=O\bigl((r^*)^{k/2}\bigr),
\]
\ie, geometric asymptotic convergence of the modified ARM with rate \(r^*<L\).
\end{thm}

\begin{proof}
We first justify this choice of $\theta = \theta^*$. 
Linearizing \(h\) about \(0\) gives the two‐step recurrence
\[
y_{k+1} 
\;\approx\;L\,\tilde y_k
\;=\;
L\bigl(y_k + \theta(y_k-y_{k-1})\bigr)
\;=\;(1+\theta)\,L\,y_k \;-\;\theta\,L\,y_{k-1}.
\]
Define the characteristic polynomial
$r^2 \;-\;(1+\theta)L\,r \;+\;\theta\,L \;=\;0 $,
with roots 
\begin{equation}
\label{e:roots}
r_{\pm} = \frac{1}{2}(1+\theta)L \pm \frac{1}{2} \sqrt{(1+\theta)^2 L^2 - 4 \theta L}. 
\end{equation}
To make the two roots coincide (and thus minimize its spectral radius), we impose the zero‐discriminant condition
\[
\bigl[(1+\theta)L\bigr]^2 \;-\;4\,\theta\,L
\;=\;0.
\]
Choosing the root \(\theta = \theta^*\) that lies in $(0,1]$ gives
$\theta^* \;=\;\frac{2 - L \;-\;2\sqrt{1-L}}{L}$.
Substituting in \eqref{e:roots}, we obtain the (double) root
\[
r^*
=\frac{1}{2} (1+\theta^*) L 
=1\,-\,\sqrt{\,1-L\,}.
\]
Since \(L \in (0,1)\), we have that 
$0 < r^* = 1-\sqrt{1-L} < L < 1$.
The recurrence equation then gives 
\[
y_{k+1} \;\approx\; (1+\theta^*)L\,y_k - \theta^*L\,y_{k-1}
\;\approx\; r^*\,y_k,
\]
as desired.
\end{proof}

Note that if $\theta_k = \frac{k-1}{k+2}$ is chosen as the parameter in the Nesterov's acceleration method, $\theta$ increases gradually as $k$ increases and $\theta_k = 1$ when $k \to 1$. For large $k$, the roots of \eqref{e:roots} becomes complex and $|r| \to \sqrt{L}$ which is greater than $L$.    

\begin{figure}[t!]
\centering
\includegraphics[width=0.48\textwidth]{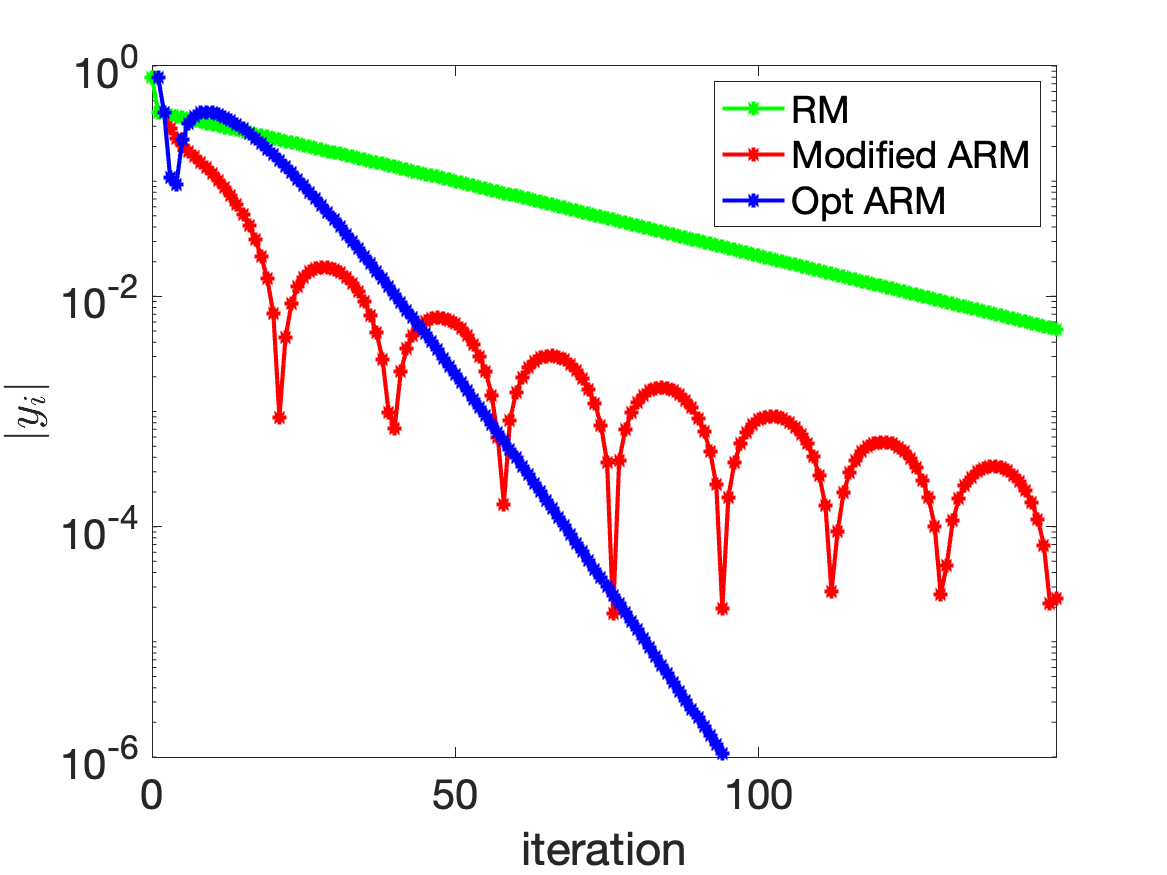} \includegraphics[width=0.48\textwidth]{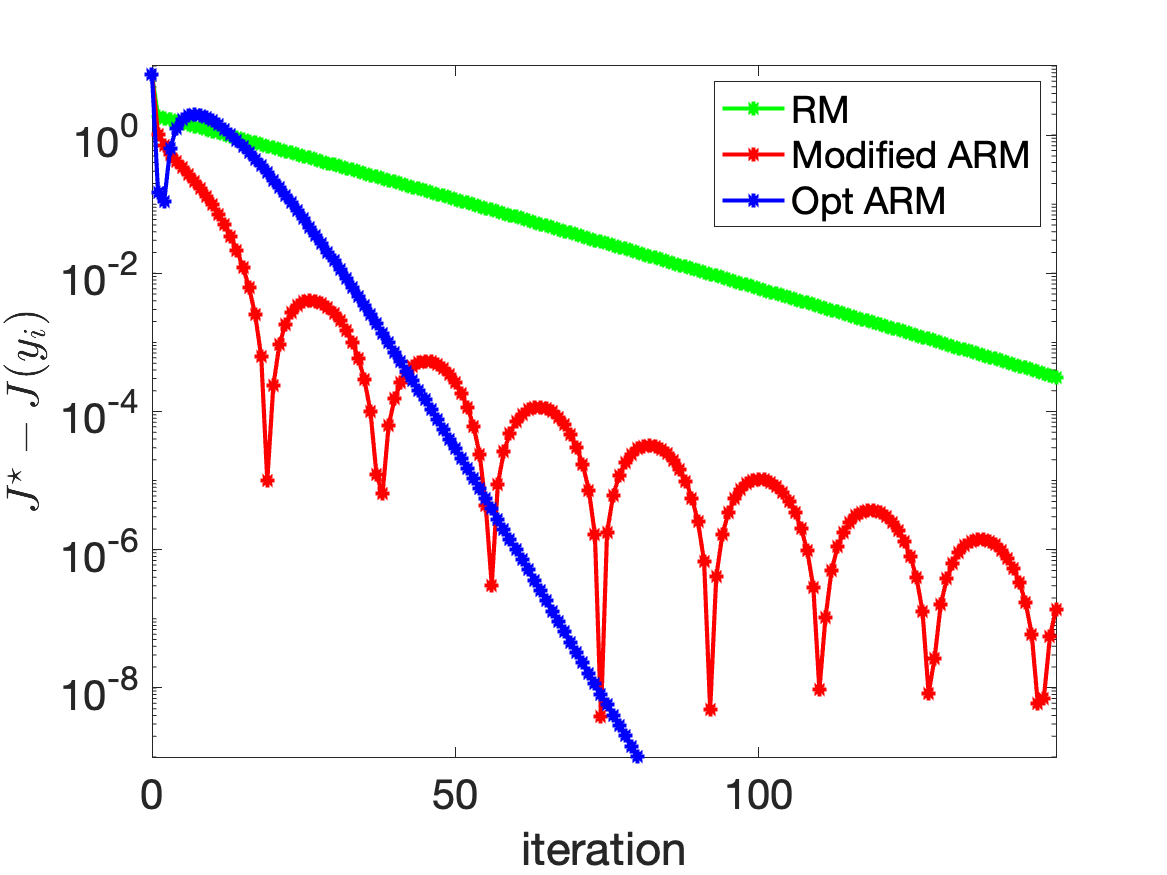} \\ 
\bigskip
\includegraphics[width=0.48\textwidth]{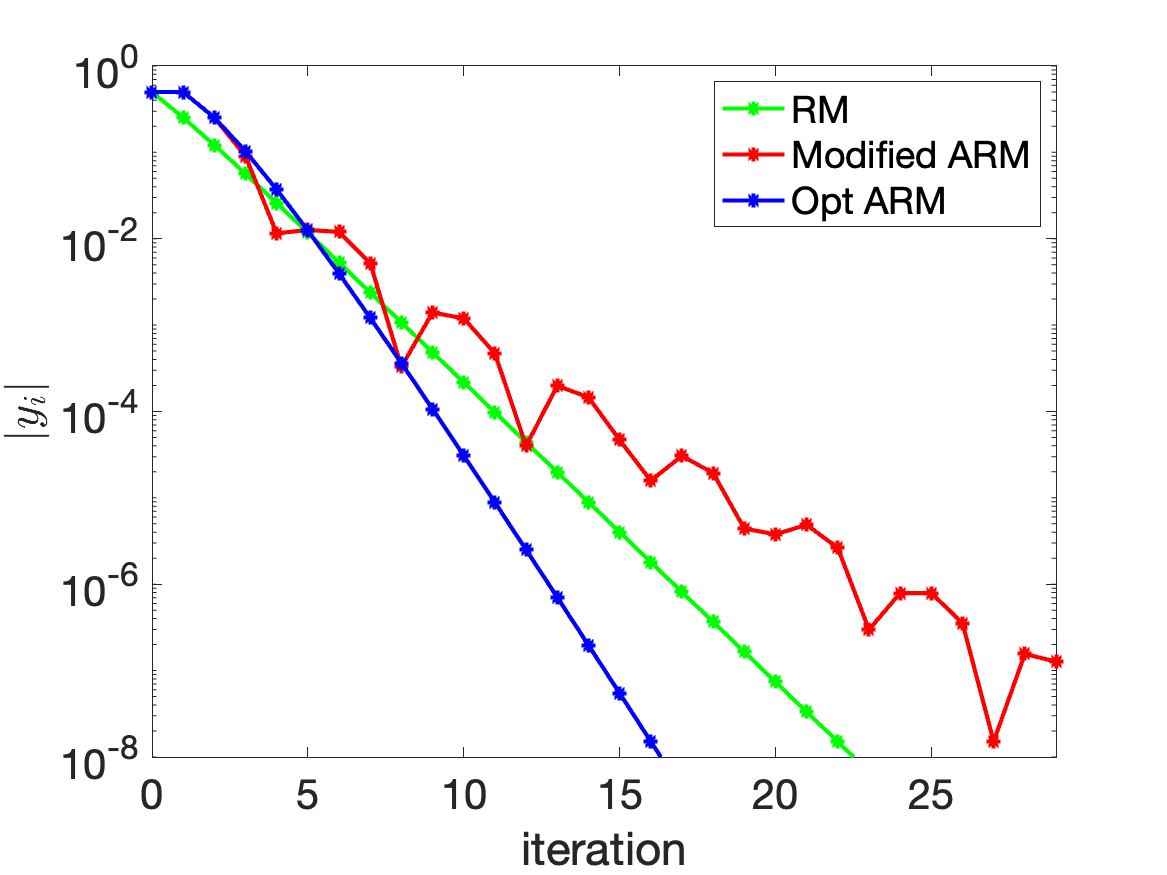} \includegraphics[width=0.48\textwidth]{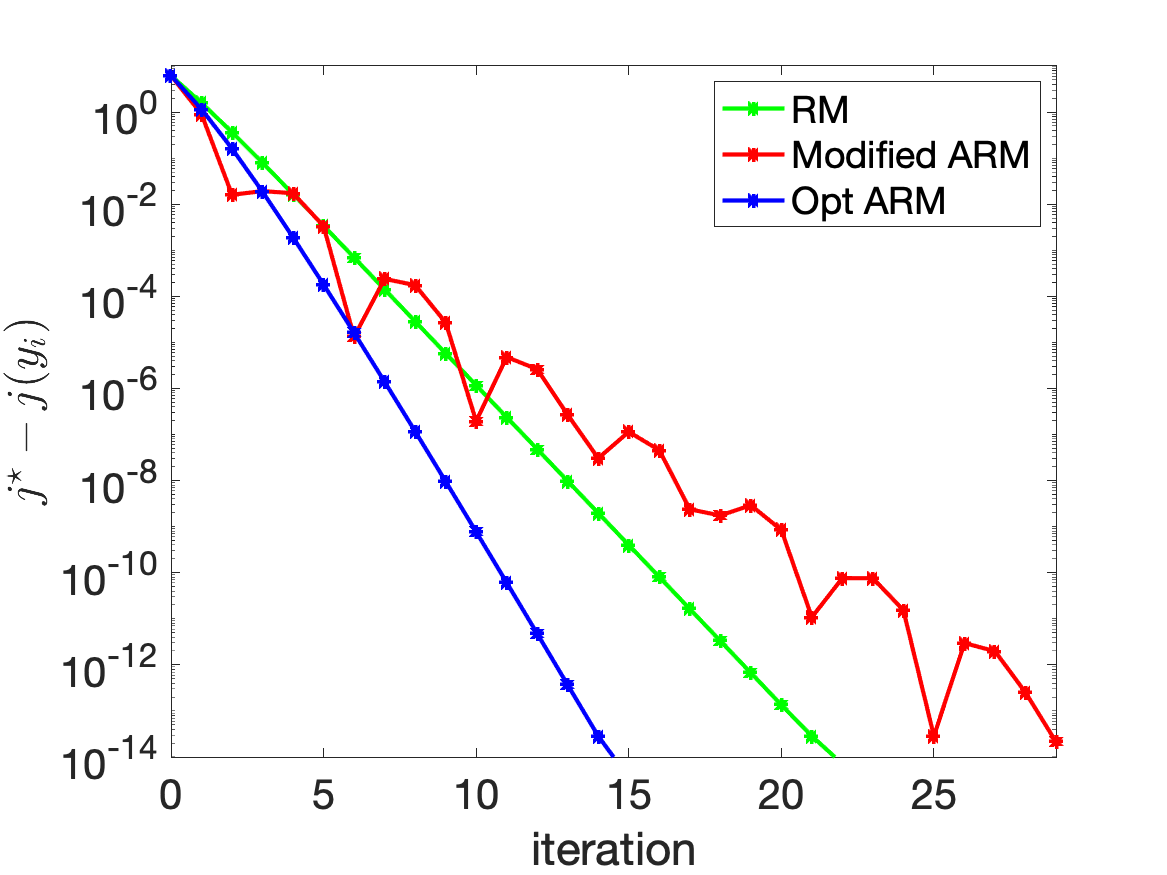}
\caption{Convergence of  $y_k$ {\bf(left)} and $J^* - J(y_k)$  {\bf (right)}  for two choices of parameters and initial conditions. 
 {\bf (top)} $ f_- = 1, f_+ = 100, \delta = 0.02$, and $f_0 = f_- + (f_+ -f_-) \chi_{(0.8-\delta,0.8+\delta)}$. 
 {\bf (bottom)} $ f_- = 1, f_+ = 10, \delta = 0.5$, and $f_0 = f_- + (f_+ -f_-) \chi_{(0.5-\delta,0.5+\delta)}$. See \cref{s:PropARMc} for details.}
\label{fig:Jconvergence}
\end{figure}

To illustrate \cref{t:ConvRate1d,t:ConvRate1d}, we perform a numerical experiment on Poisson's problem \eqref{e:PoissonOpt}.  In \cref{fig:Jconvergence}, we plot the center of the interval $y_i$ and the error $j(y_i)-j^*$. \Cref{fig:Jconvergence}(top) illustrates the high contrast and small $\delta$ regime as $f_+ / f_- = 100$ and $\delta=0.02$. In this case, the original rearrangement method (RM) is convergent with geometric rate  $L = (1-\frac{f_-}{f_+})(1-\delta) \approx 0.97$ (see the green curves). The modified accelerated rearrangement method (Modified ARM) with $\theta_k = \frac{k-1}{k+2}$ reaches a smaller error in the first few iterations as shown in the red curves. However, the error is not monotone decreasing. Instead, $y_i$ bounces toward $0$ as observed in Nesterov's acceleration method applied to other problems. The results of optimal ARM  are shown in the blue curves which demonstrate a faster convergence with $r^* \approx 0.83$.  

In \cref{fig:Jconvergence}(bottom),  we choose  $f_+ / f_- = 10$ and $\delta=0.5$. The original rearrangement method (RM) generates the optimal solution with an error less than $10^{-14}$ with $22$ iterations (See the green curves). It converges with the rate constant $L=(1-\frac{f_-}{f_+})(1-\delta)=0.45$. As observed in \cite{kao2021linear}, the arrangement method converges in a few iterations for low contrast regime and small $\delta$. It is interesting to see that the modified ARM does not accelerate the convergence and performs worse than RM. The optimal ARM method produces faster convergence results shown in blue curves with the geometric rate constant $r\approx 0.26$ and reaches an error of the objective function less than $10^{-14}$ with $15$ iterations.

\begin{algorithm}[t!]
\DontPrintSemicolon
\KwIn{A triangular mesh $\mathcal{T}_h$ of $\Omega$, 
 constants $f_+ > \overline{f} > f_- >0$, 
initial function $f_0 \in \{ f_-, f_+ \}^{|\mathcal{T}_h|}$, 
and tolerance $\varepsilon > 0$.}
\KwOut{Sequence $\{f_k\}$ with $f_k\in \{ f_-, f_+ \}^{|\mathcal{T}_h|}$ converging towards a stationary point of \eqref{e:PoissonOpt}.}
\textbf{Initialize:} Set $k_0 = 0$.\;
\For{ k = 0,1,2,\ldots }{
  \textbf{(Solve the Poisson problem.)} Using  Lagrange $P_1$ or $P_2$ finite elements (continuous piecewise‑linear or quadratic, respectively) to approximately solve $-\Delta u_k = f_k$ on $\Omega$ with $u_k = 0$ on $\partial \Omega$.\;
  \textbf{(Compute cell‐average of the Fr\'echet derivative.)} Project \( J'(f_k) = u_k\)  into the (discontinuous)  \(P_0\) space:
    \[
      \bar u_k = \mathrm{Proj}_{P_0}(u_k),
      \quad
      \bar u_k|_T = |T|^{-1} \int_T u_k\,dx,
      \quad T \in \mathcal{T}_h. 
    \]\;
    \textbf{(Momentum‐extrapolate the Fr\'echet derivative.)}  Let $\theta_k' = \frac{k - k_0}{k - k_0 +3} \in [0,1)$ and set 
    \[
      g_k = \bar u_k \;+\; \theta_{k}' \bigl(\bar u_k-\bar u_{k-1}\bigr). 
    \]\;
    \textbf{(Threshold the extrapolated Fr\'echet derivative.)}  Propose a density using the bang–bang threshold:
    \[
      \tilde f 
      = f_- + (f_+-f_-)\,\chi_{\{g_k(x)\ge c_k  \,\}},
    \]
    where \( c_k \) is chosen so that \( |\Omega|^{-1} \int_\Omega \tilde f\,dx \approx \overline{f} \).\;
    \textbf{(Assert monotonicity.)} 
    \uIf{$J(\tilde{f}) > J(f_{k})$}{
        \textbf{(Accept the proposed update.)}  Set  $$f_{k+1} = \tilde{f}.$$ \; 
    }
    \Else{
        \textbf{(Restart acceleration.)}  Update $k_0 = k$ (so that $\theta_{k}' = 0$) and take the RM update 
$$f_{k+1} = f_- + (f_+-f_-)\,\chi_{\{\bar{u}_k \ge c_k\,\}} 
$$ 
where \( c_k \) is chosen so that \( |\Omega|^{-1} \int_\Omega f_{k+1}\,dx \approx \overline{f} \).\;
    }
    \textbf{(Check convergence.)} If $\|f_{k+1} - f_k\|_{L^2(\Omega)} \leq \varepsilon$, STOP. \;
}
\caption{Restarted Accelerated Rearrangement Method (RARM) for the extremal Poisson problem \eqref{e:PoissonOpt}.}
\label{a:ARM}
\end{algorithm}

\section{Computational Results}
\label{s:CompRes}
In this section, we consider the practical performance of the RM \eqref{e:RM}, ARM \eqref{e:ARM}, and RARM \eqref{e:RARM} for solving general two-phase composite optimization problems of the form \eqref{e:gen2phase} and, in particular, the extremal Poisson problem \eqref{e:PoissonOpt} and the minimal principal eigenvalue problem for the weighted Dirichlet-Laplacian \eqref{e:EigOpt}. 
We first describe some computational details for our implementation (\cref{s:CompMeth}) and then consider several computational experiments demonstrating and comparing the behavior of these methods (\crefrange{s:CompRes2DPoisson}{s:CompCuboid}).

\subsection{Computational details for the  accelerated rearrangement method} \label{s:CompMeth}
In \cref{a:ARM}, we describe the numerical implementation of the RARM for the extremal Poisson  problem \eqref{e:PoissonOpt}. In brief, at each iteration, we solve the Poisson Equation with a given $f$ using the finite element method, then rearrange $f$ according to an extrapolated Fr\'echet derivative of the objective. 
The RM is the same as \cref{a:ARM}, except we take $\theta_k = 0$ for every $k\geq 0$. 
The ARM (without restarting) is the same as \cref{a:ARM}, except we take $k_0=0$ at every iteration $k\geq 0$ so that $\theta'_k = \theta_k = \frac{k}{k+3}$.  

To modify the method for other optimization problems of the form \eqref{e:gen2phase}, one must change the PDE and the Fre\'echet derivative. For example, for the extremal eigenvalue problem \eqref{e:EigOpt}, the momentum-extrapolation step  is made based on the projected values of $J'(f) = \lambda_1 u_1^2$. In general, computing the Fr\'echet derivative may require solving an adjoint PDE. For instance, consider the non-energetic optimal control problem 
\( J(f) := \int_\Omega j(x,u) \), 
where \( j(x,u) \) is convex in \( u \) and \( u=u_f \) solves \( -\Delta u = f \) with \( u = 0 \) on \( \partial \Omega \). Then the derivative is given by \( J'(f) = w \), where \( w\) solves the adjoint equation \( -\Delta w = \frac{\partial j}{\partial u}(x,u_f) \) in \( \Omega \) with homogeneous Dirichlet boundary conditions \cite{chambolle2025stability}.

We implement \cref{a:ARM} to solve \eqref{e:PoissonOpt} in \verb+Python+ and use the \verb+FEniCS+ software library to solve the Poisson Equation via the finite element method with Lagrange $P_1$ or $P_2$ finite elements. To generate domain meshes, we either use the built-in \verb+FEniCS+ meshes for regular domains or \verb+mshr+ to generate meshes for irregular domains. 
In the thresholding step, we identify the set $D = \{g(x) \ge c  \,\}$ so that \( c \) satisfies 
\( |\Omega|^{-1} \int_\Omega f_- + (f_+-f_-) \chi_D \,dx \approx \overline{f} \), as follows. 
Let $\pi$ be a permutation sorting the values of $g$ in descending order, so that $g_{\pi(1)}\ge g_{\pi(2)}\ge\cdots\ge g_{\pi(N)}$. Define the cumulative volume
$S_k=\sum_{j=1}^k |T _{\pi(j)}|$,
and let $\mathcal K$ be the largest index with $\frac{S_\mathcal K}{|\Omega|} \leq \delta $. We then take $D$ to be the union of the mesh elements that have $\mathcal K$ largest values of $g$, 
$D = \cup \{ T_i \colon i\in\{\pi(1),\ldots,\pi(\mathcal K)\} \}$.

To solve \eqref{e:EigOpt}, we make the following modifications to the implementation described above. The method is again implemented in \verb+Python+, using the  \verb+FEniCS+  software library to construct the Lagrange $P_2$ stiffness and mass matrices. The eigenvalue problem is solved using the \verb+krylov-schur+ solver implemented in the \verb+SLEPc+ software library. We threshold based on  the momentum-extrapolated projected values of $J'(f) = -\lambda_1 u_1^2$. 
    
All of the following numerical experiments were performed on a 2022 laptop computer with an Apple M2 processor and 16GB of memory.

\subsection{The extremal Poisson problem \eqreftitle{e:PoissonOpt} on a rectangle} \label{s:CompRes2DPoisson}
We consider the extremal Poisson problem \eqref{e:PoissonOpt} on the rectangle $\Omega = [0,2]\times [0,1]$. We take a uniform $512\times 256$ grid, generate a regular triangular mesh with 262,144 cells, and use Lagrange $P_2$ finite elements. 
We choose parameters $f_-=1$, $f_+=10$, and $\delta = 0.2$. 
The results are displayed in \cref{fig:2dPoissonConvergence}.
We observe that the RM objective function values are monotonic and converge at iteration 74 to a value of 2.579523.
The ARM objective function values are not monotonic (they display the 'bouncing' behavior that is characteristic for accelerated methods) and converge at iteration 60 to a value of 2.579524. 
Finally, RARM restarts at iteration 13, so that the objective function values are monotonic and converge at iteration 36 to a value of 2.579524.  
In all three cases, the iterations terminate  because $\|f_{k} -f_{k-1}\| = 0$.

\subsection{The extremal Dirichlet-Laplacian eigenvalue problem \eqreftitle{e:EigOpt} on a rectangle} \label{s:CompRes2Deig}
We consider the extremal eigenvalue problem
\eqref{e:EigOpt} on the two-dimensional rectangle $\Omega = [0,2]\times [0,1]$. We take a uniform $128\times 64$ grid, generate a triangular mesh with 16,384 cells, and use Lagrange $P_1$ finite elements. 
We choose parameters $f_-=1$, $f_+=10$, and $\delta = 0.2$. The results are displayed in \cref{fig:2dEigConvergence}.
We observe that the RM objective function values are monotonic and converge at iteration 55 to a value of 1.7600.
The ARM objective function values are not monotonic and converge at iteration 36 to a value of 1.7600. 
Finally, RARM restarts at iteration 13, so that the objective function values are monotonic and converge at iteration 14 to a value of 1.7600.  
In all three cases, the iterations terminate  because $\|f_{k} -f_{k-1}\| = 0$.

\subsection{The extremal Poisson problem \eqreftitle{e:PoissonOpt} on a starfish shaped domain} \label{s:CompStarfish}

We consider the extremal Poisson problem \eqref{e:PoissonOpt} on a punctured starfish shaped domain of the form 
$$
\Omega = \big\{ (r,\theta) \colon \frac{4}{10} < r< 1 + \frac{3}{10}\cos(5 \theta) \big\},
$$
 with $f_- = 1$, $f_+ = 10$, and $\delta = 0.25$. 
We use \verb+meshr+ to generate a triangular mesh in $\Omega$ with 111,821 triangles and use Lagrange $P_2$ finite elements. 
Note that, unlike the rectangular grid discussed in \cref{s:CompRes2DPoisson}, the grid in this example is not regular. 

The results of the experiment are shown in \cref{fig:Starfish}. 
We observe that the RM objective function values converge at iteration 63 to a value of 1.795670. 
Both the ARM and RARM converge at iteration 34 to a value of  1.795671. 
Both ARM and RARM take fewer iterations than RM to reach approximately the same value. 
Note that because the mesh is irregular, the values for RM and RARM are only approximately monotonic; during the threshold step, the exact fraction of mesh cells cannot be chosen to satisfy the constraint $\frac{|D|}{|\Omega|} = 0.25$, so the constraint is only approximately satisfied.

\subsection{The extremal Poisson problem \eqreftitle{e:PoissonOpt} on a cut Gourat surface} \label{s:CompResGourat}
Define the quartic polynomial
$
\phi(x,y,z) = x^4 + y^4 + z^4 \;-\;(x^2 + y^2 + z^2)\;+\;\tfrac{9}{20}$. 
The zero‐level set \(\{\phi=0\}\) is a smooth closed surface, known as a \emph{Gourat surface}.  
To introduce a boundary, we remove a small ``corner'' patch of the surface in the first octant. More precisely, define
$R \;=\; \left\{(x,y,z)\in \mathbb R^3 \colon x\geq\frac{1}{2}, \; y\geq \frac{1}{2}, \; z\geq\frac{1}{2}\right\}$.
We then define the \emph{cut Gourat surface} by 
$$
\Sigma = \{(x,y,z)\in \mathbb R^3 \colon \phi(x,y,z) = 0 \,\, \textrm{and} \, \, (x,y,z) \notin R\}.
$$ 
The boundary of $\Sigma$ has three connected components. 
A plot of $\Sigma$ is given in \cref{fig:CompResGourat}. 

To construct $\Sigma$ numerically, 
we sample \(\phi\) on a $150^3$ uniform grid in \([-1.2,1.2]^3\) and 
use the marching‐cubes algorithm implemented in \verb+skimage.measure+ to generate a triangular surface mesh with 220,082 cells. 
We then use FEniCS to solve the Poisson problem
\begin{align*}
-\Delta_{S} u &= f  && \text{on } \Sigma \\
u &= 0 &&  \partial \Sigma    
\end{align*}
using Lagrange $P_2$ finite elements
Here, $\Delta_S$ is the Laplace–Beltrami operator. 
The extremal Poisson problem is then defined as in \eqref{e:PoissonOpt}. 
We choose $f_- = 1$, $f_+ = 10$, and $\delta = 0.25$.

The results are shown in \cref{fig:CompResGourat} and \cref{fig:CompResGourat2}. Note that we observe symmetry breaking; the optimal density $f_*$ does not share the triangular symmetry of the domain.  
We observe that the RM objective function values converge at iteration 99 to a value of $1.102362\times10^3$. 
ARM converges at iteration 51 to a value of  $1.102351\times 10^3$. 
Finally, RARM converges at iteration 12 to a value of 
$1.101330\times 10^3$. 
Here, we see that RARM converges early. ARM takes significantly fewer iterations to reach approximately the same value as RM. 

This example illustrates that the proposed methods naturally extend to more general settings, such as elliptic PDEs on Riemannian manifolds.

\subsection{The extremal Poisson \eqreftitle{e:PoissonOpt} on a cuboid} \label{s:CompCuboid}
We consider  \eqref{e:PoissonOpt} on the cuboid $\Omega = [0,2]\times [0,1] \times [0,1]$. 
We generate a tetrahedral mesh with 324,000 cells  use Lagrange $P_1$ finite elements.
We choose parameters $f_-=1$, $f_+=10$, and $\delta = 0.28$. 
The results are displayed in \cref{fig:Cuboid}.
We observe that the RM objective function values are monotonic and converge at iteration 44 to a value of $J = 2.217184$.
The ARM objective function values are not monotonic and converge at iteration 30 to a value of $J = 2.217185$. 
Finally, RARM restarts at iteration 9, so that the objective function values are monotonic and converge at iteration 14 to a value of $J=2.217185$.  

\section{Discussion and Future Directions}
\label{s:Disc}
In this paper, we propose an accelerated rearrangement method (ARM) \eqref{e:ARM} and restarted ARM \eqref{e:RARM} to solve a class of nonconvex two-phase composite optimization problems and apply them to find the maximal total displacement or work energy of a membrane \eqref{e:PoissonOpt}, as well as the optimal density function of a membrane that minimizes its first eigenvalue (fundamental frequency) \eqref{e:EigOpt}. We prove that RARM has the property that, under mild assumptions, every weak-$*$ accumulation point of its iterates lies in the admissible class and satisfies the first‐order necessary condition for optimality.
In one dimension, we derive the convergence rate for the ARM and show that it indeed converges faster due to a smaller geometric rate constant. Estimating sharp convergence rates analytically in higher dimensions is challenging, especially when the domain is complex, such as being multiply-connected or having high genus. Nevertheless, our numerical results consistently show faster convergence for the newly proposed ARM and its restarted variant across many examples.

It would be interesting to study how ARM methods perform in other applications. In population dynamics optimization problems, it is important to find the best ways to control population size by allocating resources, adjusting birth and death rates, or designing treatments. In photonic crystal optimization problems, the goal is to design periodic structures with desired optical properties, such as maximizing the width of spectral gaps or minimizing losses in waveguides. Since solving the governing equations in these applications is often more computationally demanding, developing efficient and robust ARM methods could significantly improve the performance in finding optimal solutions for these problems.

\subsection*{Acknowledgments}
The authors would like to thank the International Centre for Mathematical Sciences (ICMS) in Edinburgh, Scotland for hosting a Research-in-Groups program to work on this project together at the birthplace of James Clerk Maxwell.

\bibliographystyle{siam}
\bibliography{refs}

\begin{figure}[p!]
\begin{center}
\includegraphics[width=0.45\textwidth]{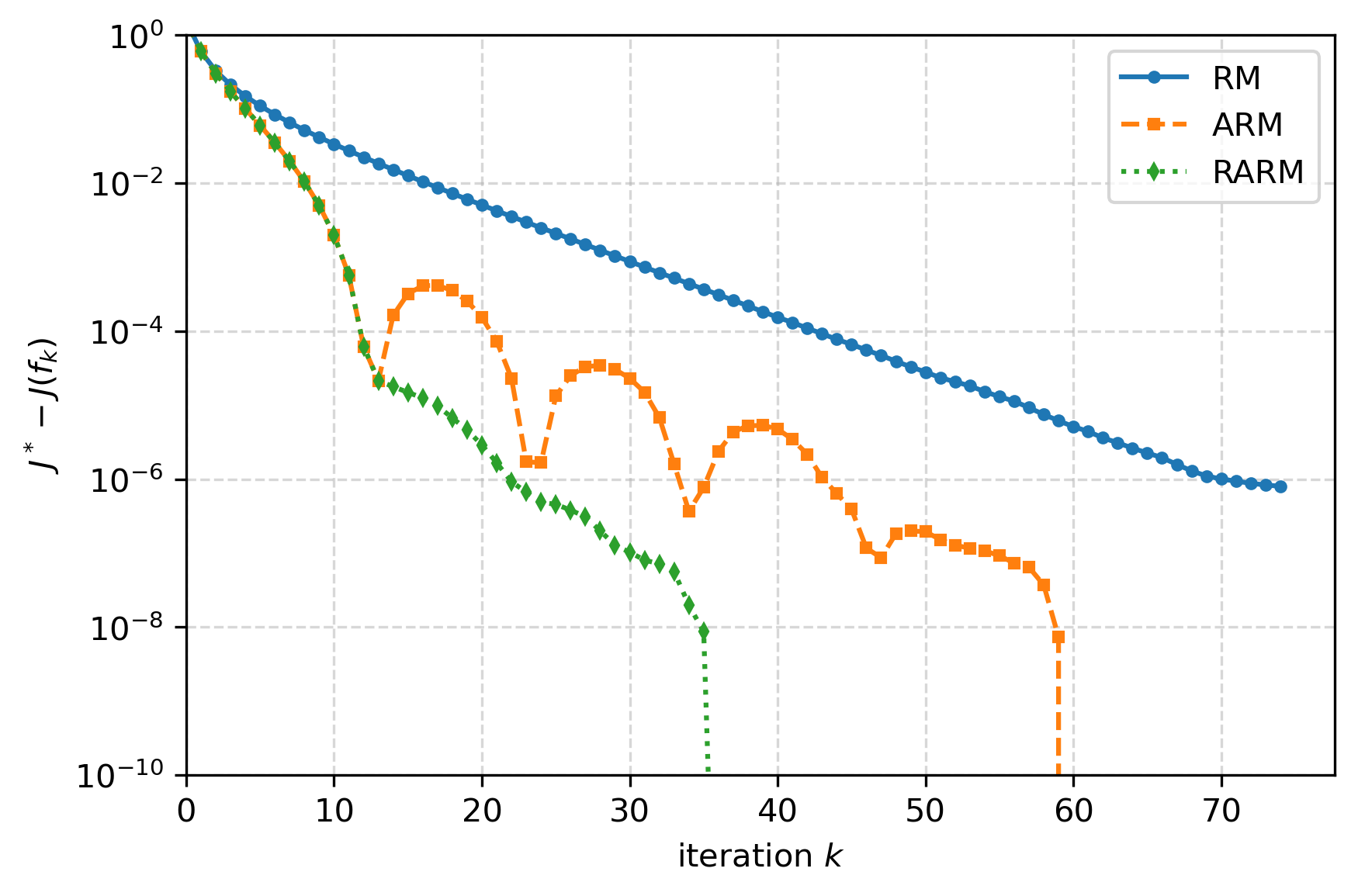}
\includegraphics[width=0.45\textwidth]{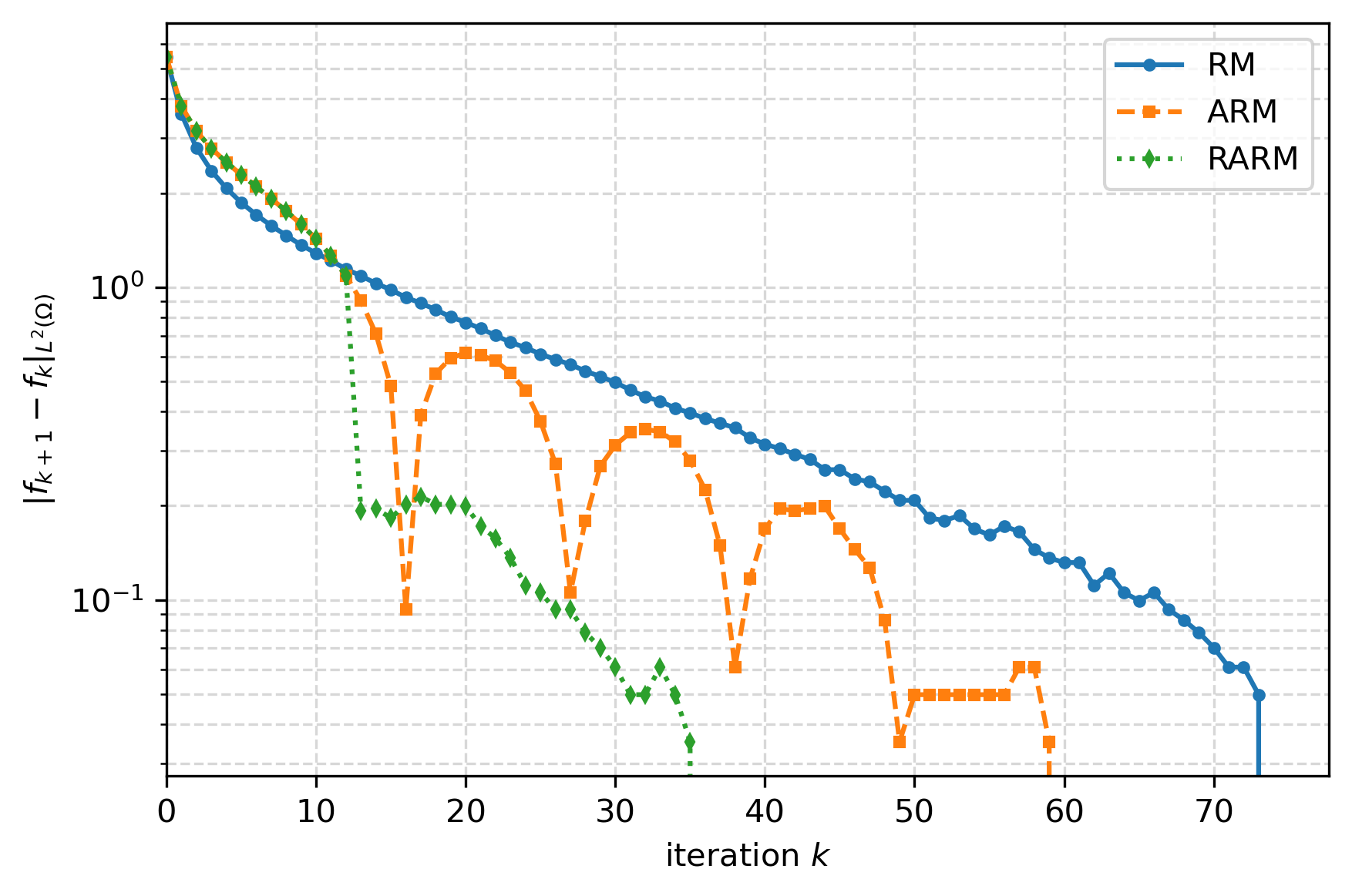} \\ 
\medskip
\includegraphics[height=0.23\textwidth]{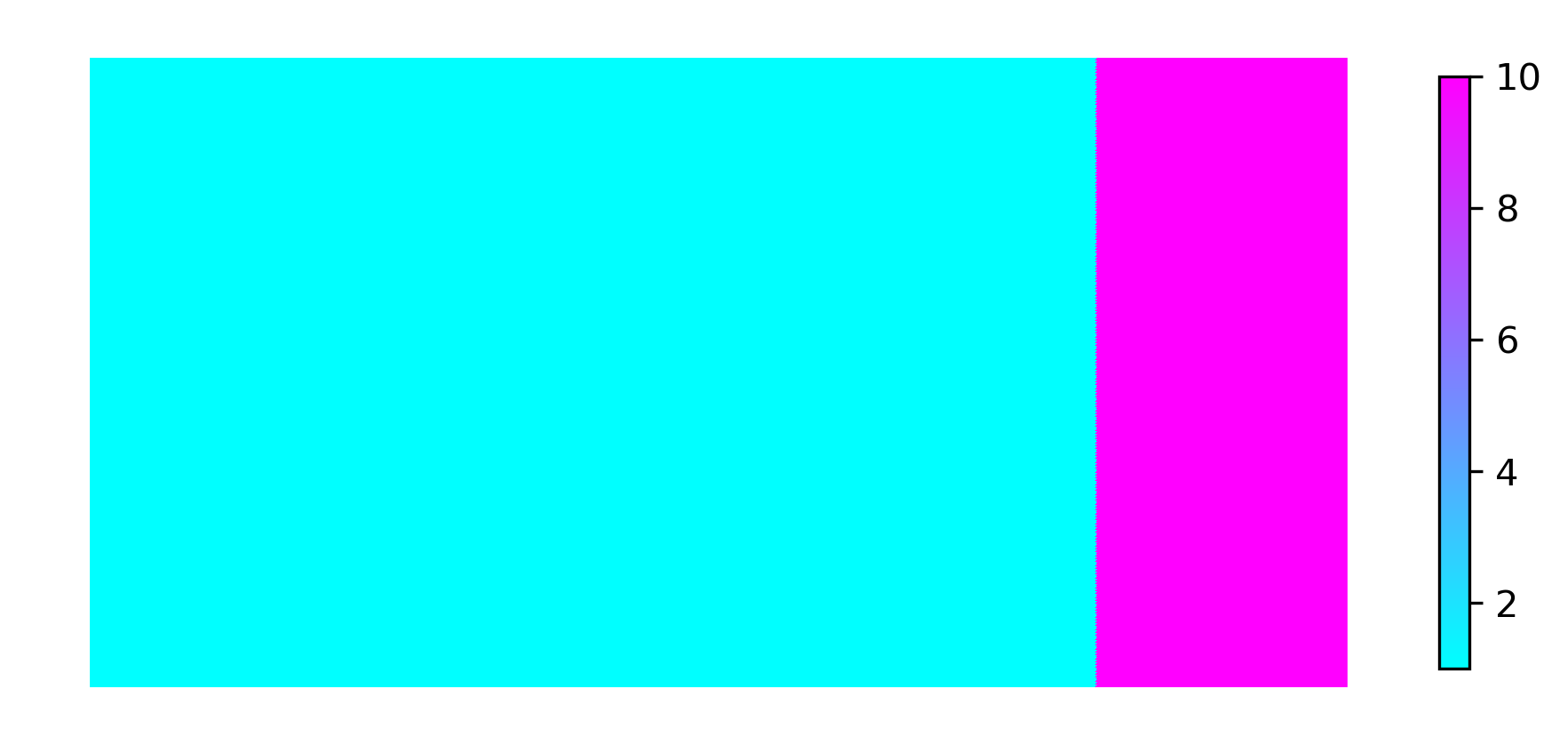}
\includegraphics[height=0.23\textwidth]{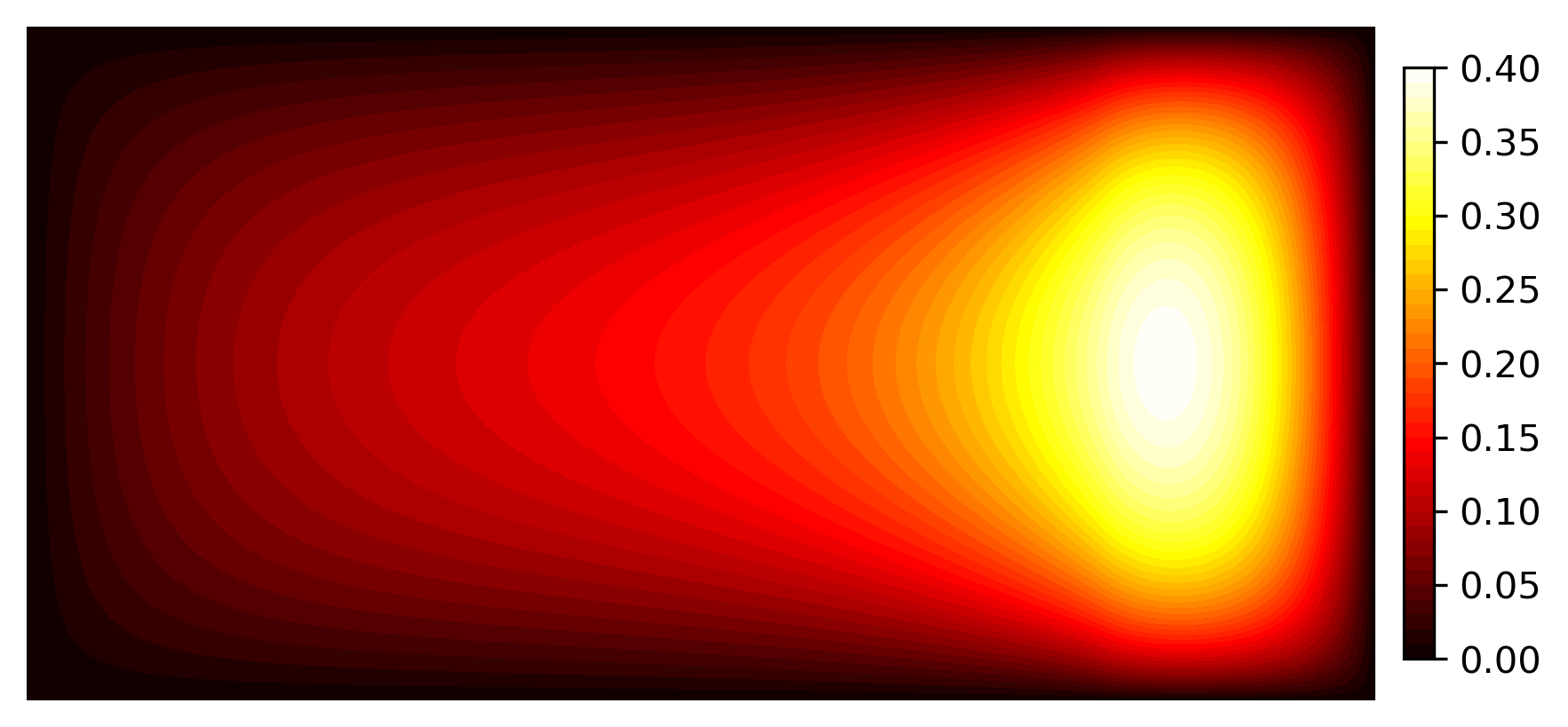} \\ 
\medskip
\includegraphics[height=0.23\textwidth]{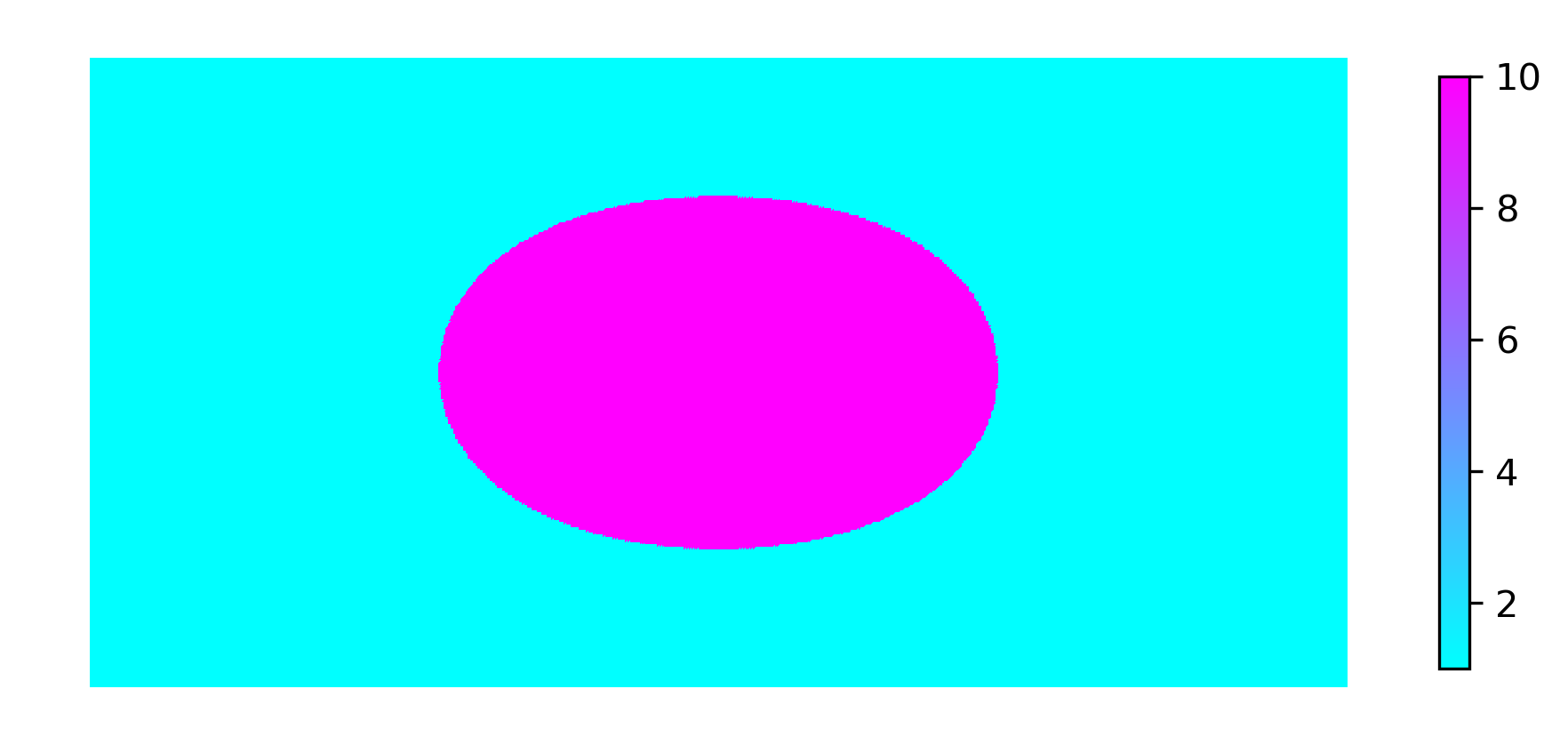}
\includegraphics[height=0.23\textwidth]{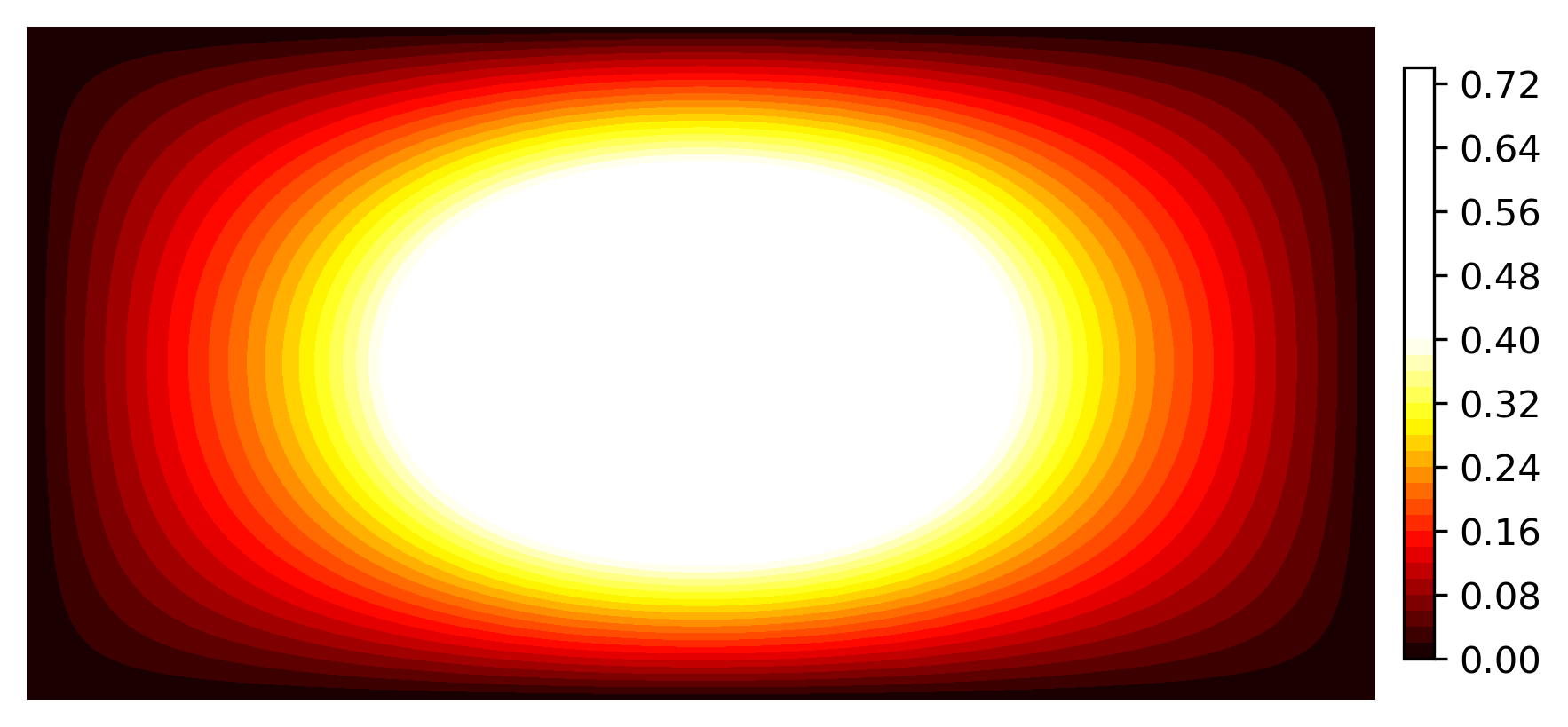} 
\end{center}
\caption{We consider  \eqref{e:PoissonOpt} on the rectangle $\Omega = [0,2]\times[0,1]$ with  $f_- = 1$, $f_+ = 10$, and $\delta = 0.2$. 
We plot the iteration $k$ vs. $J^* - J(f_k)$ {\bf (top left)} and $\|f_{k+1} - f_{k}\|_{L^2(\Omega)}${\bf (top right)} for the 
rearrangement method (RM) \eqref{e:RM},  
accelerated RM (ARM) \eqref{e:ARM}, and 
restarted ARM (RARM) \eqref{e:RARM}. 
In the middle panels, we plot the initial density $f_0$ {\bf (middle left)} and initial solution  $u_0$ {\bf (middle right)}. 
In the lower panels, we plot the final density $f^*$ {\bf (bottom left)} and final solution  $u^*$ {\bf (bottom right)}. 
See \cref{s:CompRes2DPoisson} for details.}
\label{fig:2dPoissonConvergence}
\end{figure}

\begin{figure}[t!]
\begin{center}
\includegraphics[width=0.45\textwidth]{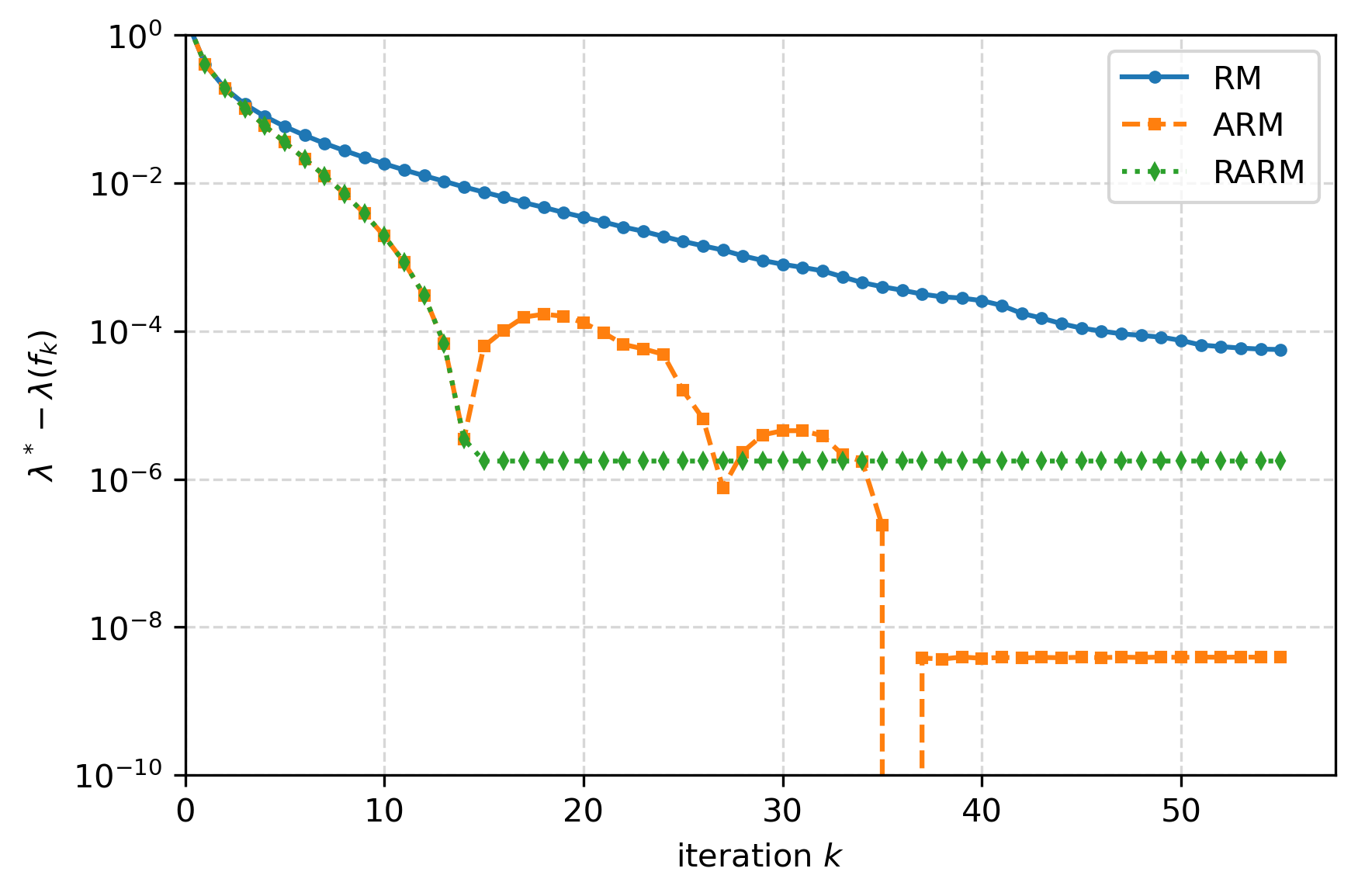}
\includegraphics[width=0.45\textwidth]{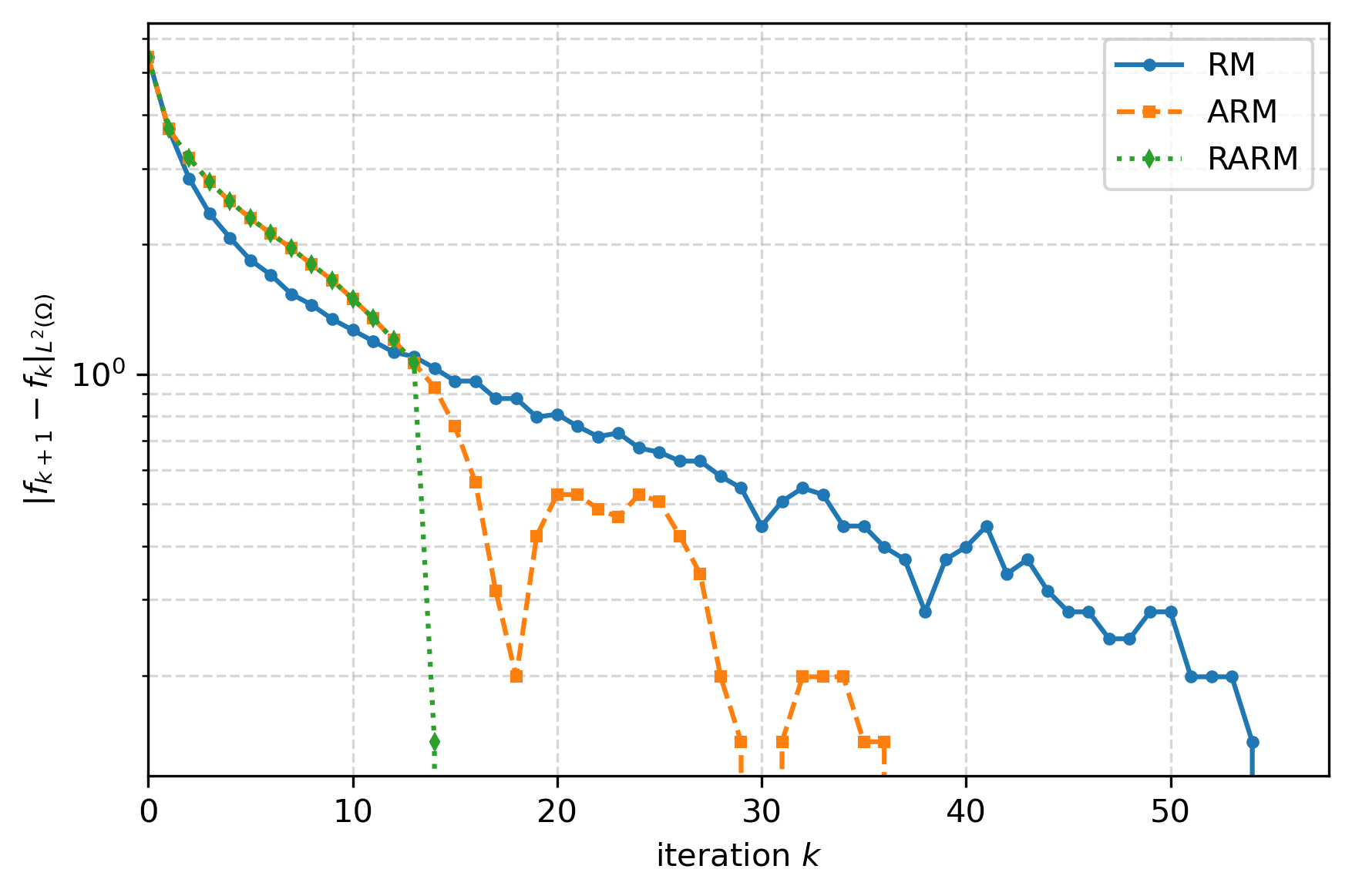}\\
\medskip
\includegraphics[height=0.23\textwidth]{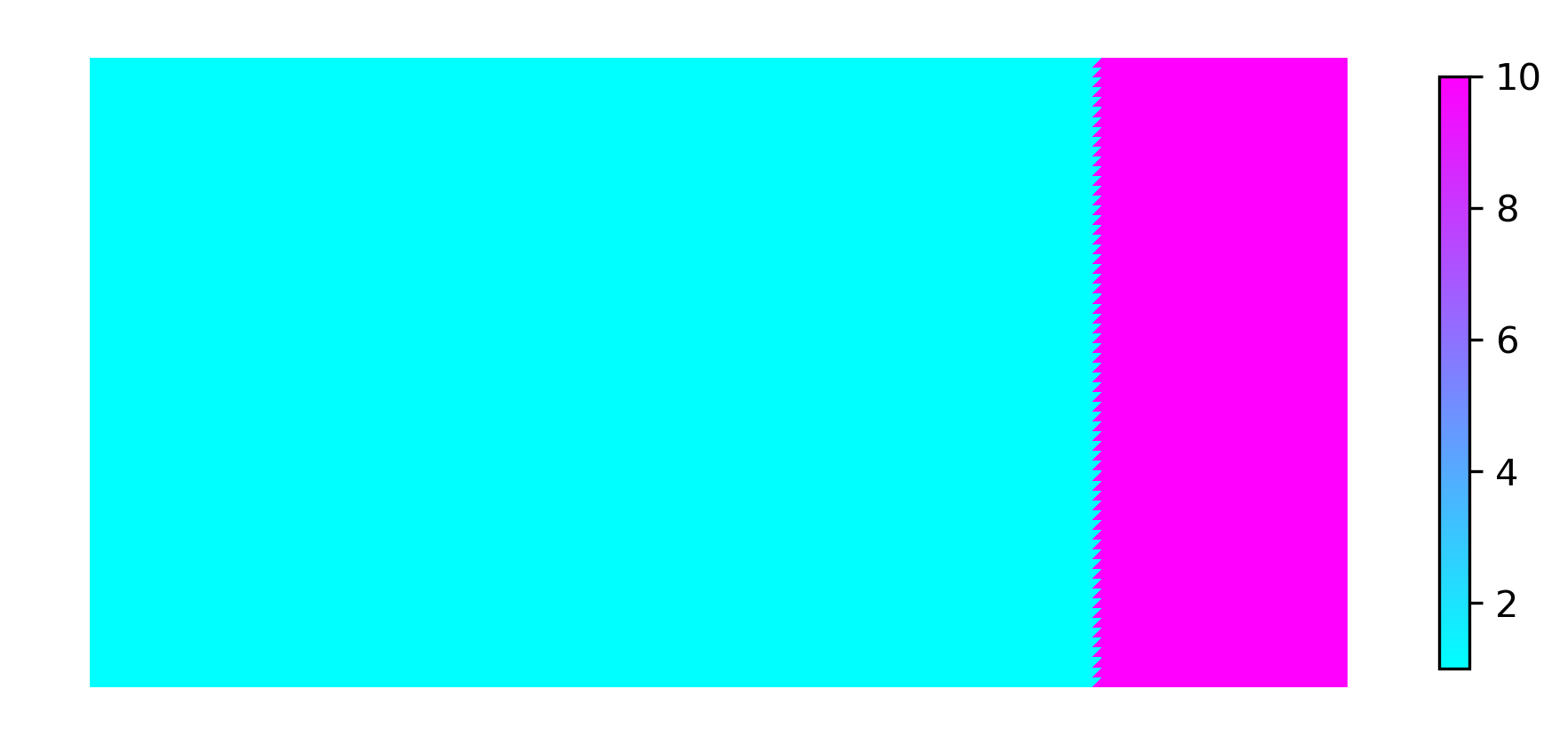}
\includegraphics[height=0.23\textwidth]{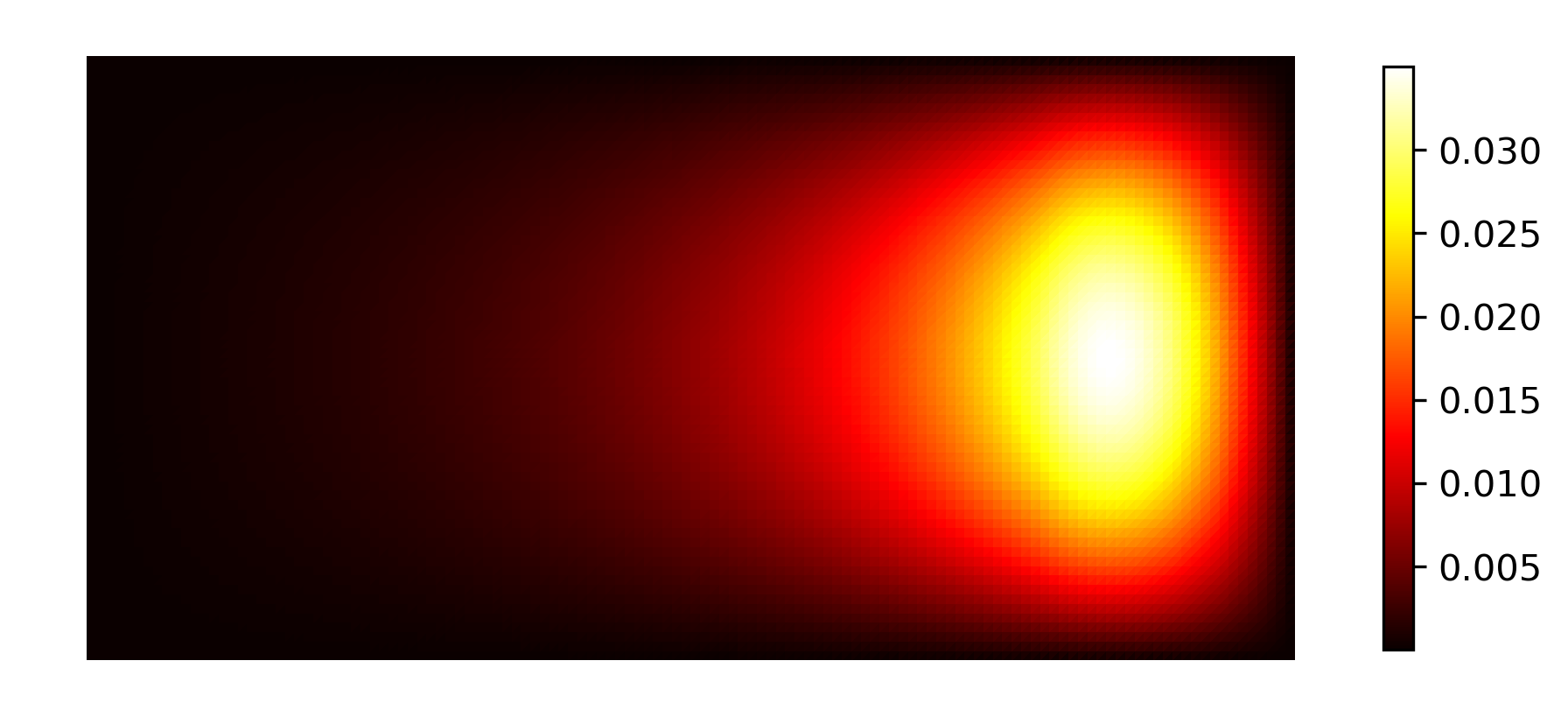} \\ 
\medskip
\includegraphics[height=0.23\textwidth]{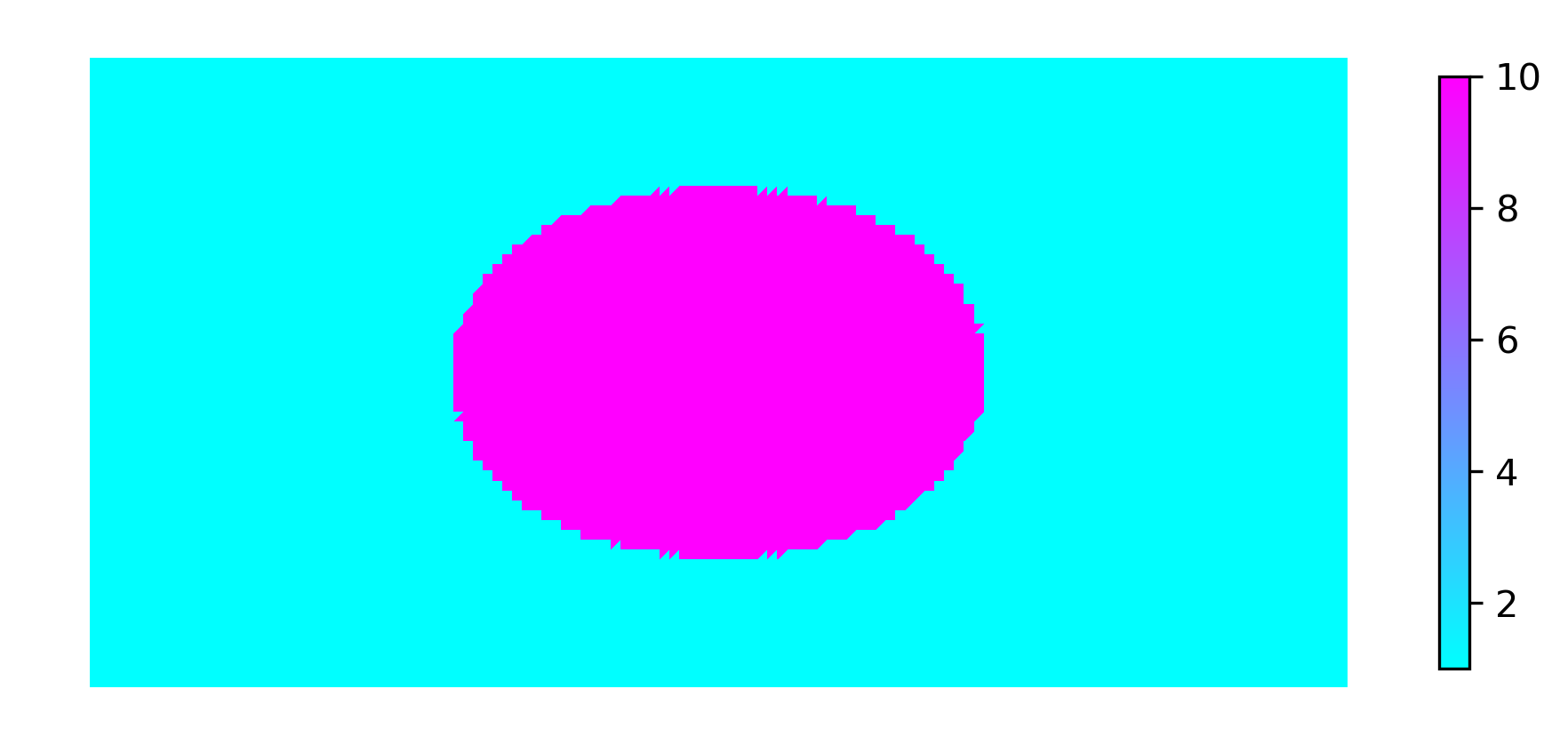}
\includegraphics[height=0.23\textwidth]{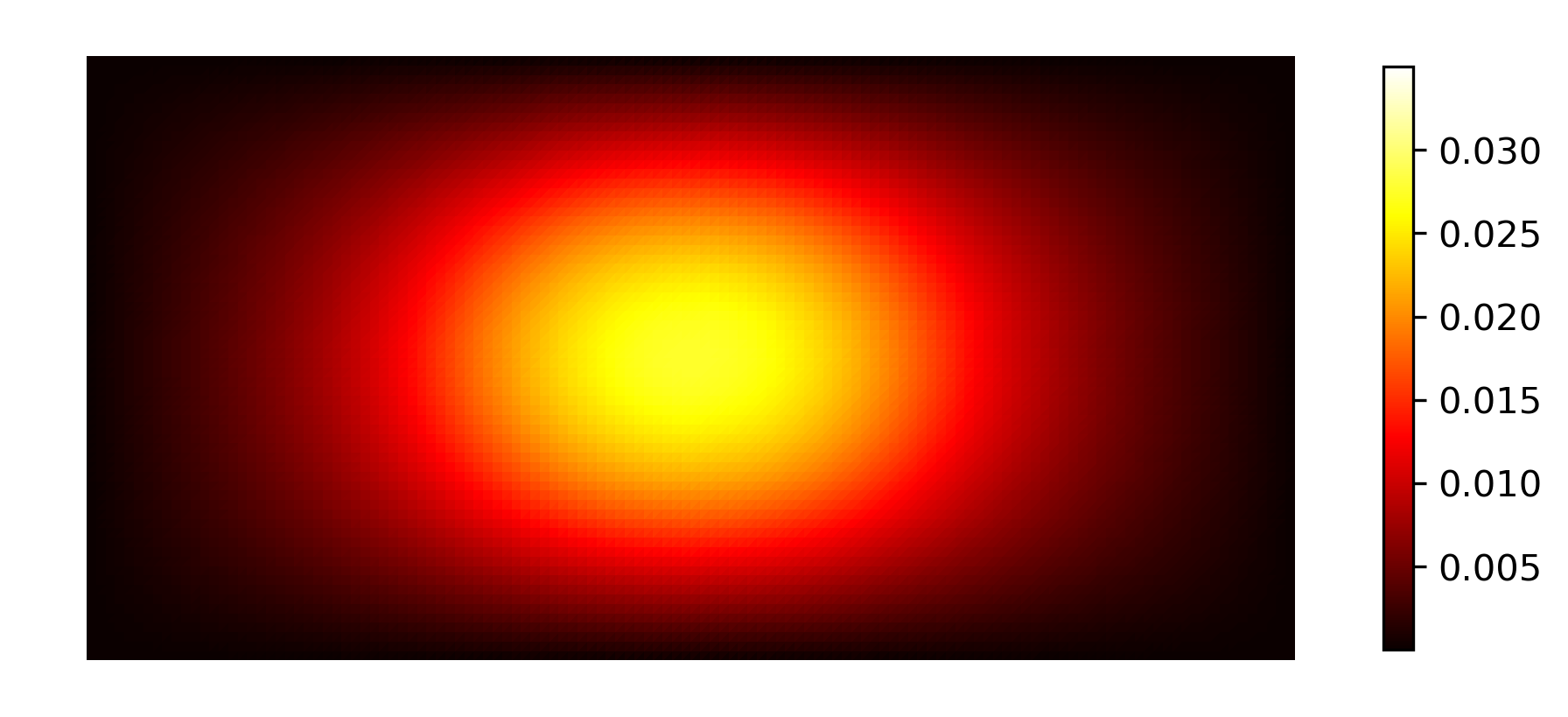}  
\end{center}
\caption{We consider the extremal eigenvalue problem \eqref{e:EigOpt} on the rectangle  $\Omega = [0,2]\times[0,1]$ with $f_- = 1$, $f_+ = 10$, and $\delta = 0.2$. 
We plot the iteration $k$ vs. $\lambda_1^* - \lambda_1(f_k)$ {\bf (top left)} and $\|f_{k+1} - f_{k}\|_{L^2(\Omega)}${\bf (top right)} for the 
rearrangement method (RM) \eqref{e:RM},  
accelerated RM (ARM) \eqref{e:ARM}, and 
restarted ARM (RARM) \eqref{e:RARM}.  
In the middle panels, we plot the initial density $f_0$ {\bf (middle left)} and initial principal eigenfunction  $u_0$ {\bf (middle right)}. 
In the lower panels, we plot the final density $f^*$ {\bf (bottom left)} and final principal eigenfunction  $u^*$ {\bf (bottom right)}. 
See \cref{s:CompRes2Deig} for details.}
\label{fig:2dEigConvergence}
\end{figure}

\begin{figure}[t!]
\begin{center}
\includegraphics[width=0.45\textwidth]{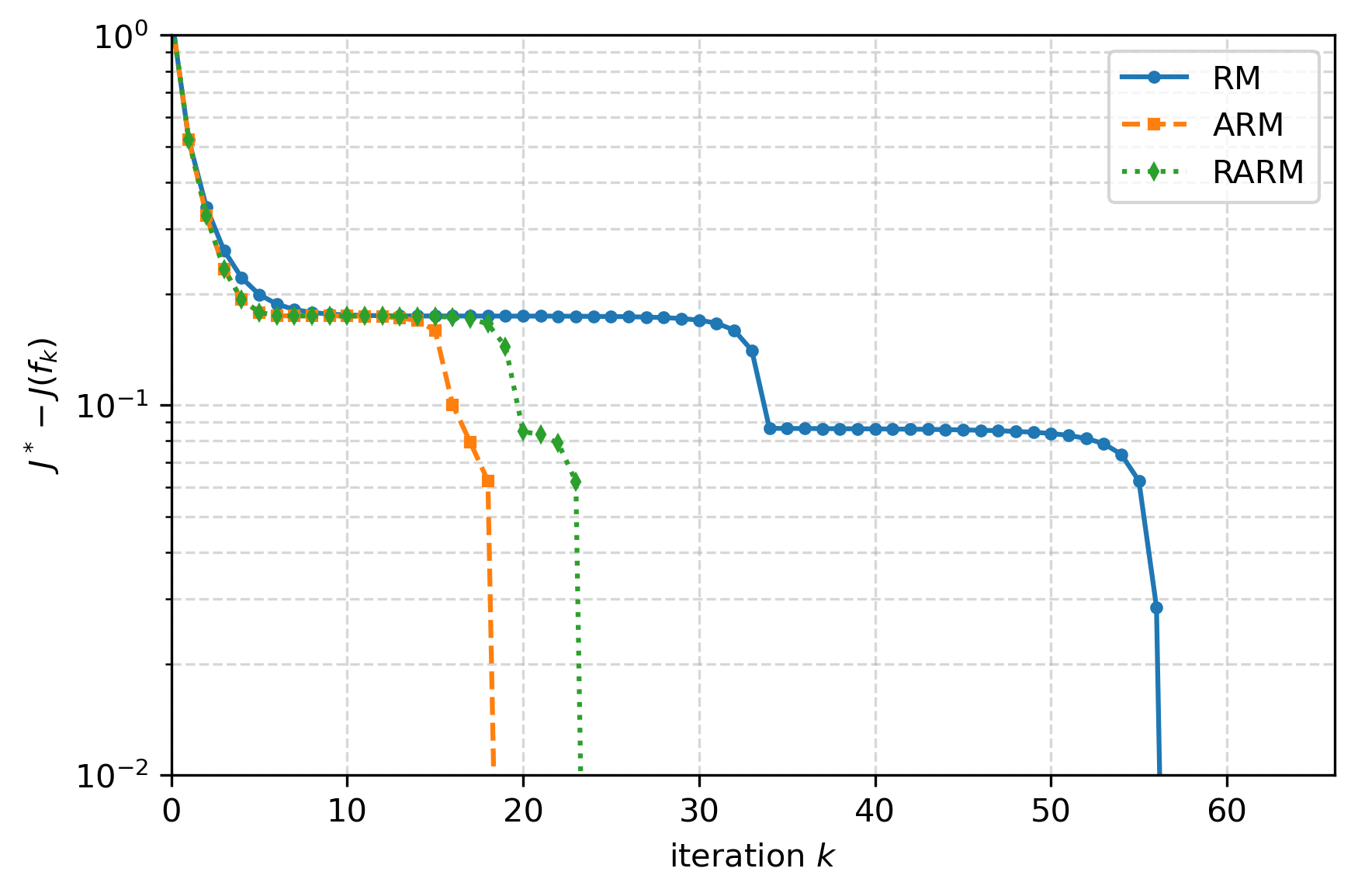}
\includegraphics[width=0.45\textwidth]{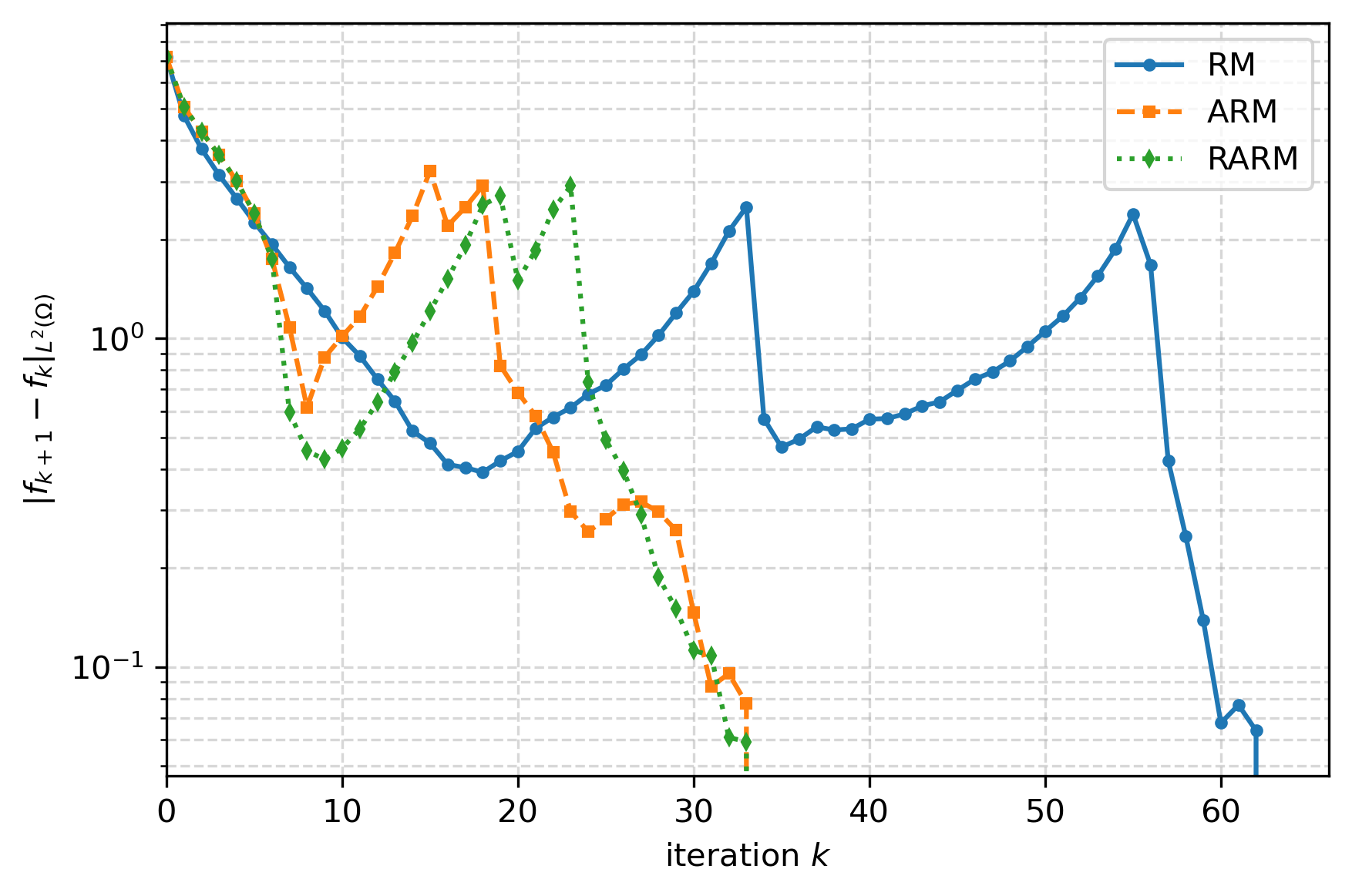} \\ 
\medskip
\includegraphics[height=0.40\textwidth]{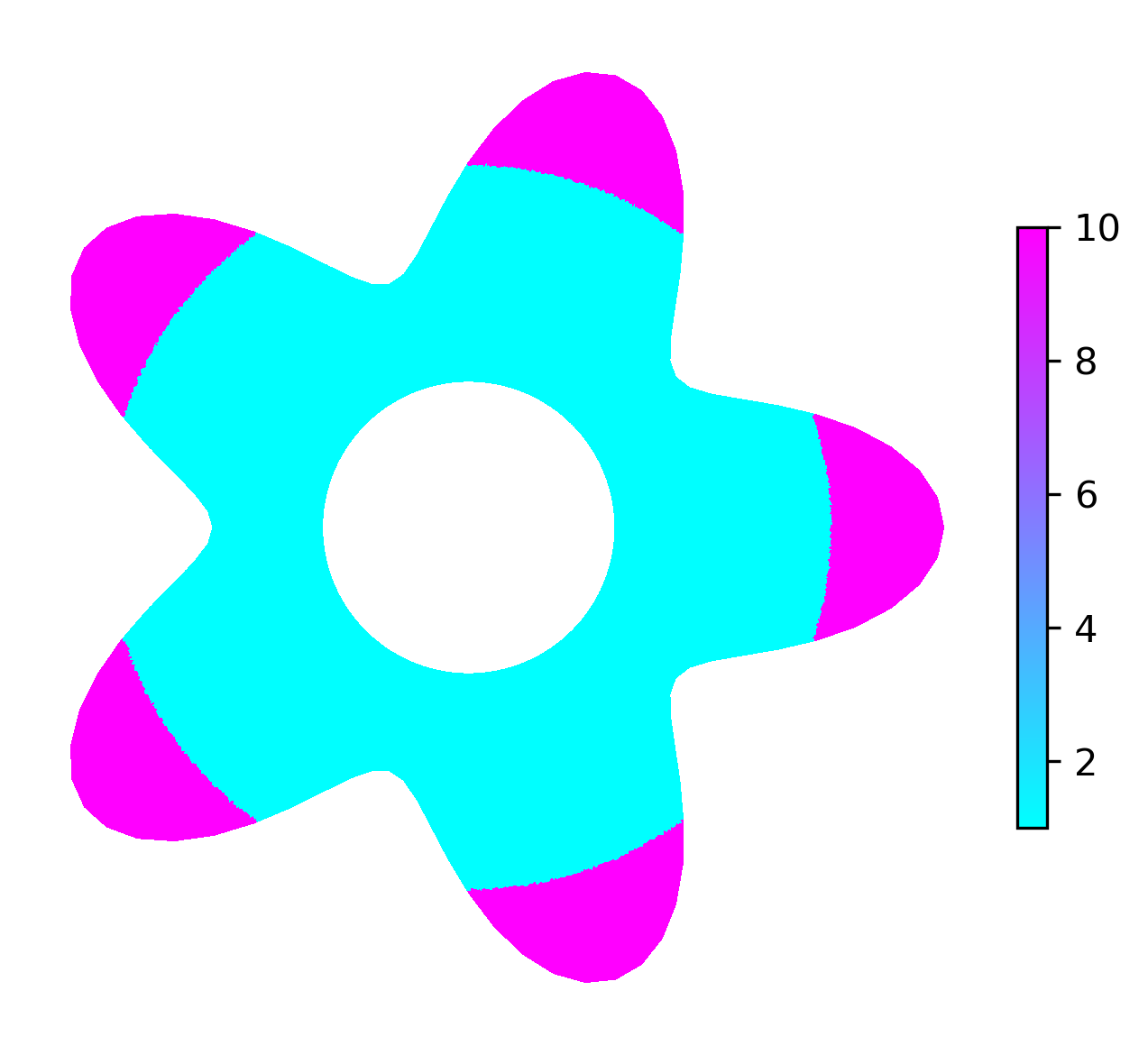}
\includegraphics[height=0.40\textwidth]{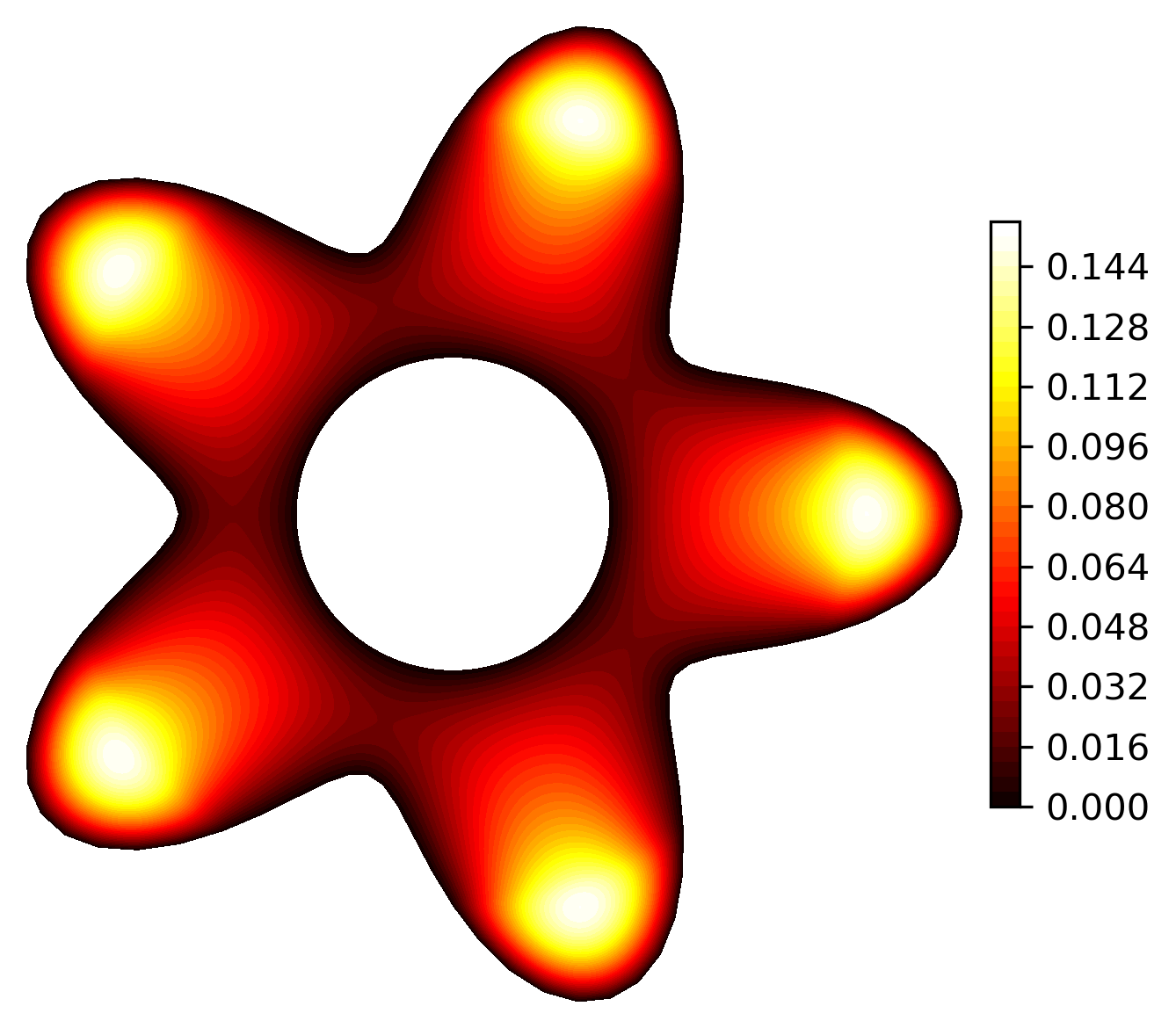} \\ 
\medskip
\includegraphics[height=0.40\textwidth]{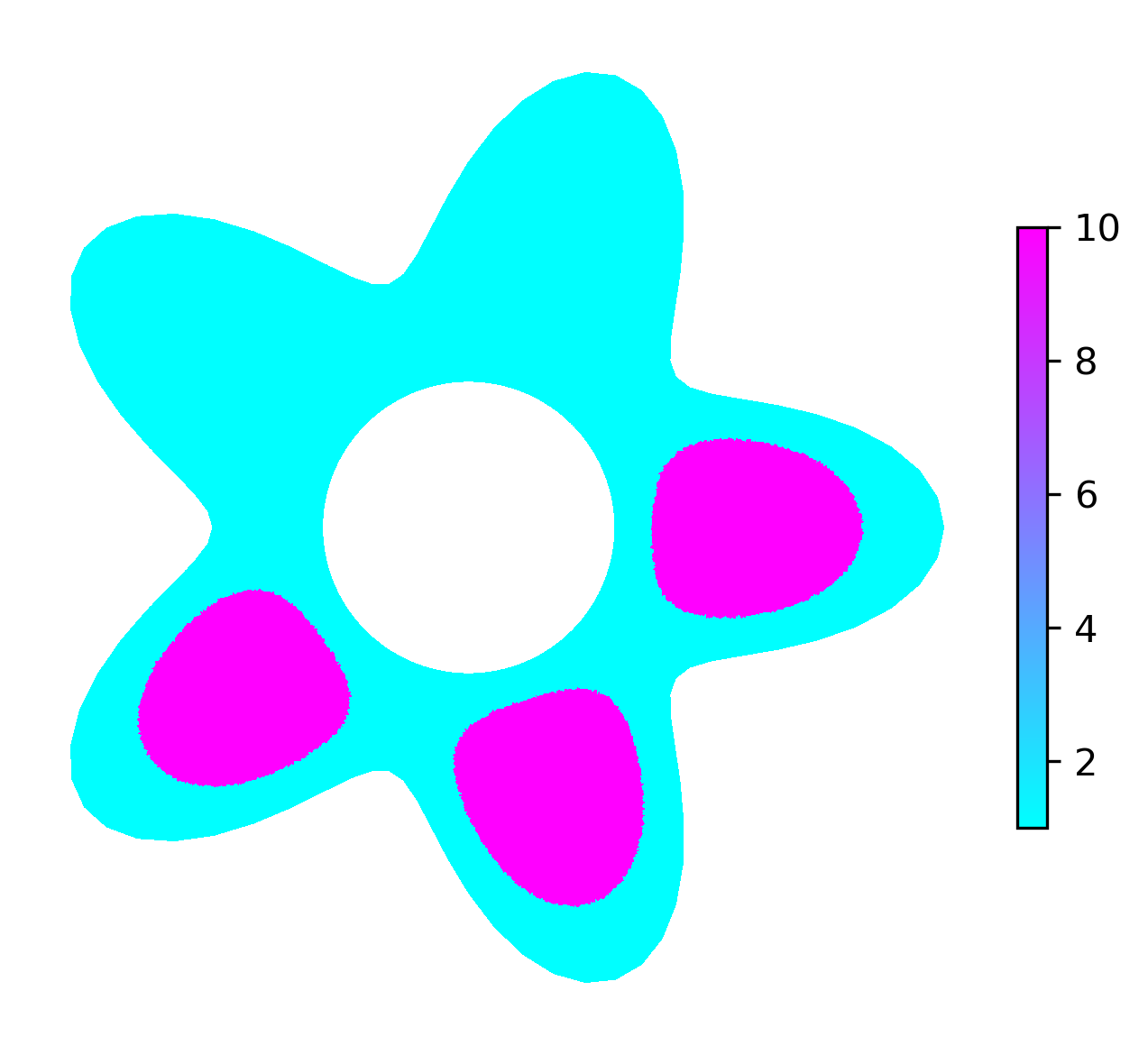}
\includegraphics[height=0.40\textwidth]{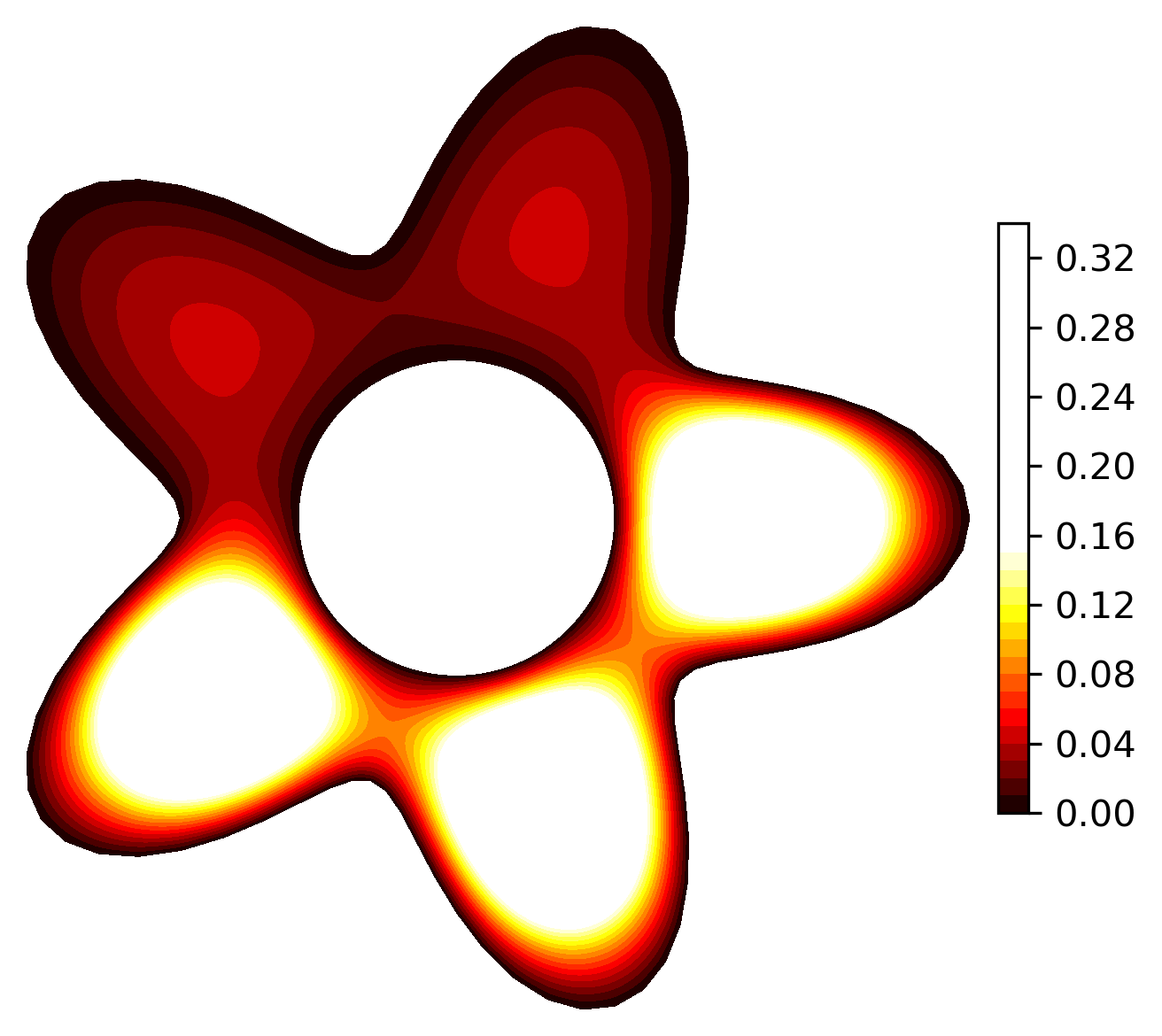} 
\end{center}
\caption{We consider  \eqref{e:PoissonOpt} on a punctured starfish domain with  $f_- = 1$, $f_+ = 10$, and $\delta = 0.25$. 
We plot the iteration $k$ vs. $J^* - J(f_k)$ {\bf (top left)} and $\|f_{k+1} - f_{k}\|_{L^2(\Omega)}${\bf (top right)} for the 
rearrangement method (RM) \eqref{e:RM},  
accelerated RM (ARM) \eqref{e:ARM}, and 
restarted ARM (RARM) \eqref{e:RARM}. 
In the middle panels, we plot the initial density $f_0$ {\bf (middle left)} and initial solution  $u_0$ {\bf (middle right)}. 
In the lower panels, we plot the final density $f^*$ {\bf (bottom left)} and final solution  $u^*$ {\bf (bottom right)}. 
See \cref{s:CompStarfish} for details.}
\label{fig:Starfish}
\end{figure}

\begin{figure}[p!]
\begin{center}
\includegraphics[width=0.44\textwidth,trim={8cm 3cm 8cm 4cm},clip]{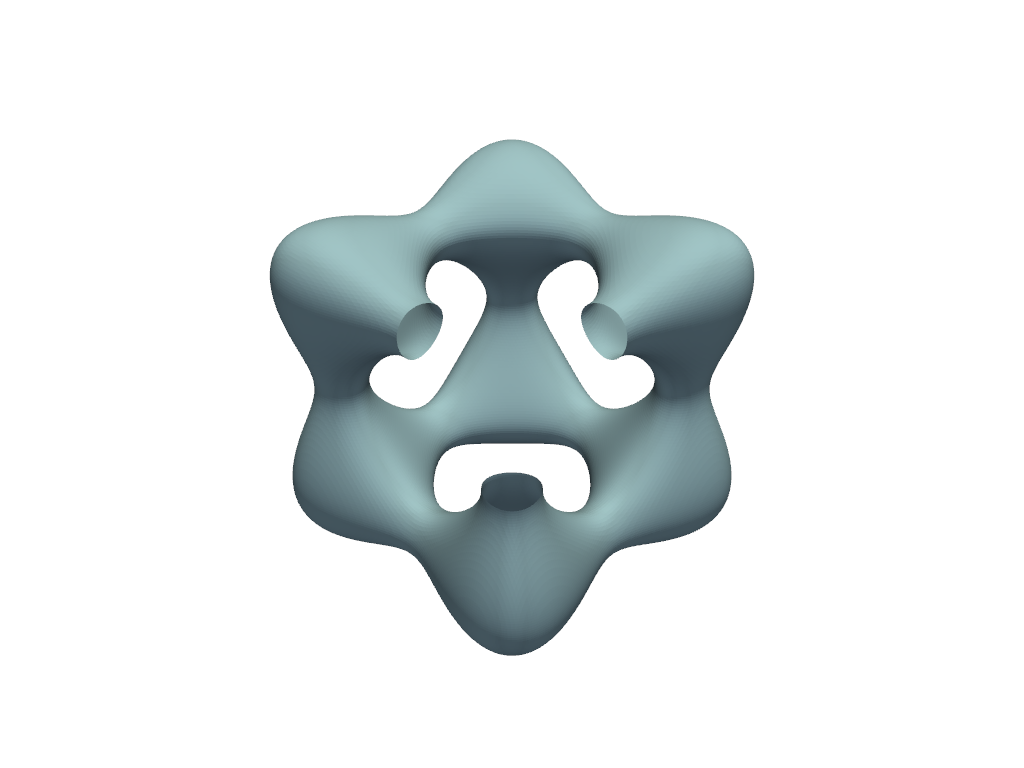}
\includegraphics[width=0.44\textwidth,trim={8cm 3cm 8cm 4cm},clip]{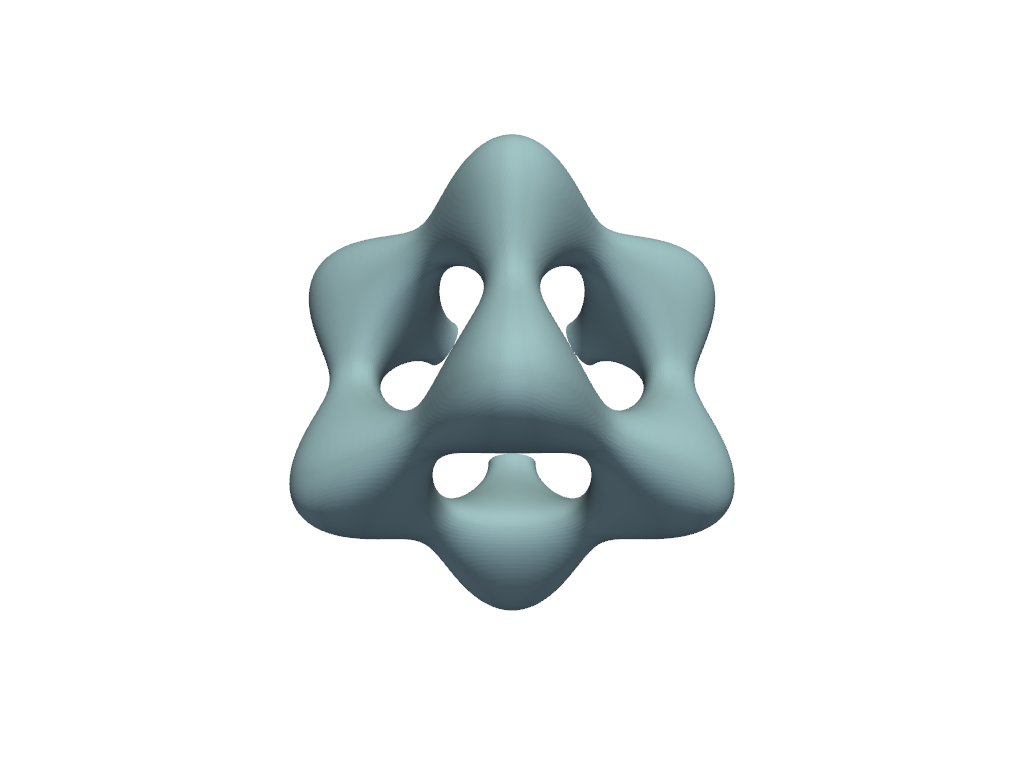} \\
\bigskip
\includegraphics[width=0.45\textwidth]{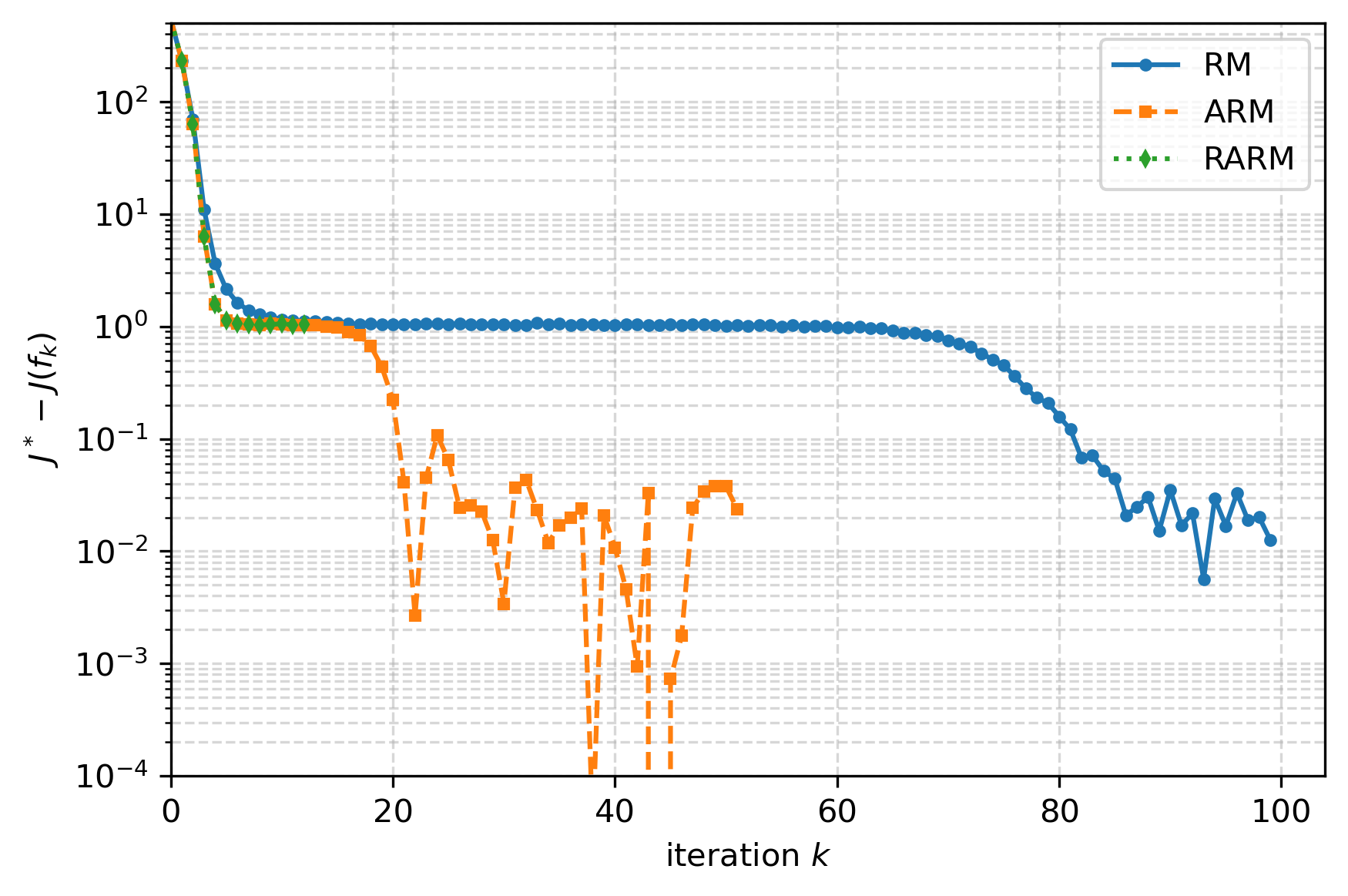}
\includegraphics[width=0.45\textwidth]{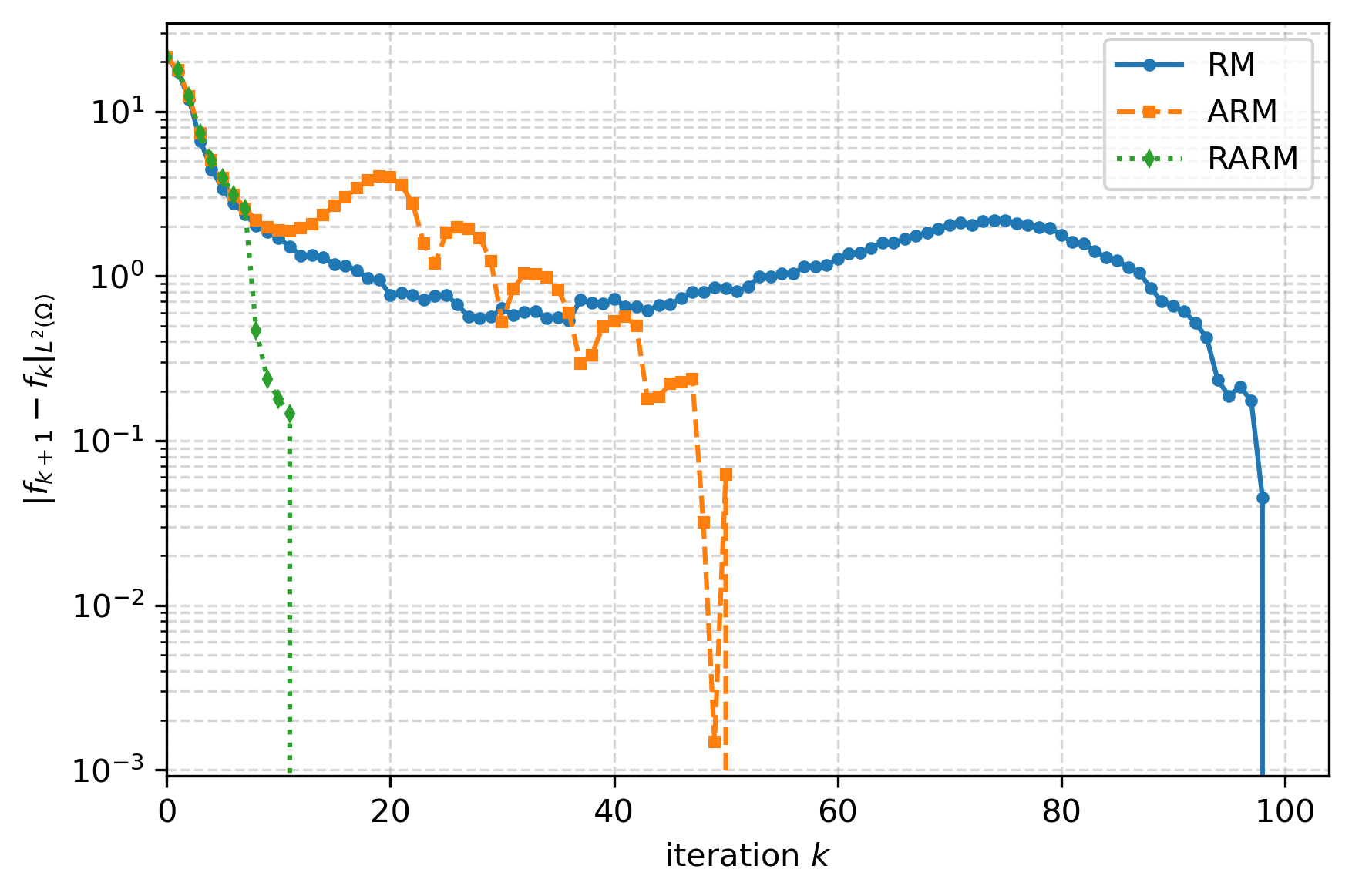}
\end{center}
\caption{ We consider  problem \eqref{e:PoissonOpt} on the cut Gourat surface with  $f_- = 1$, $f_+ = 10$, and $\delta = 0.25$. 
{\bf (top)} Two views of the cut Gourat surface looking towards the origin from $(x,y,z) = (6, 6, 6)$ {\bf (left)} and $(-4,-4,-4)$  {\bf (right)}.
{\bf (bottom)} We plot the iteration $k$ vs. $J^* - J(f_k)$ {\bf (left)} 
and $\|f_{k+1} - f_{k}\|_{L^2(\Omega)}${\bf (right)} for the 
rearrangement method (RM) \eqref{e:RM},  
accelerated RM (ARM) \eqref{e:ARM}, and 
restarted ARM (RARM) \eqref{e:RARM}. 
See \cref{s:CompResGourat} for details.}
\label{fig:CompResGourat}
\end{figure}

\begin{figure}[p!]
\begin{center}
\includegraphics[width=0.36\textwidth,trim={8cm 0 0 4cm},clip]{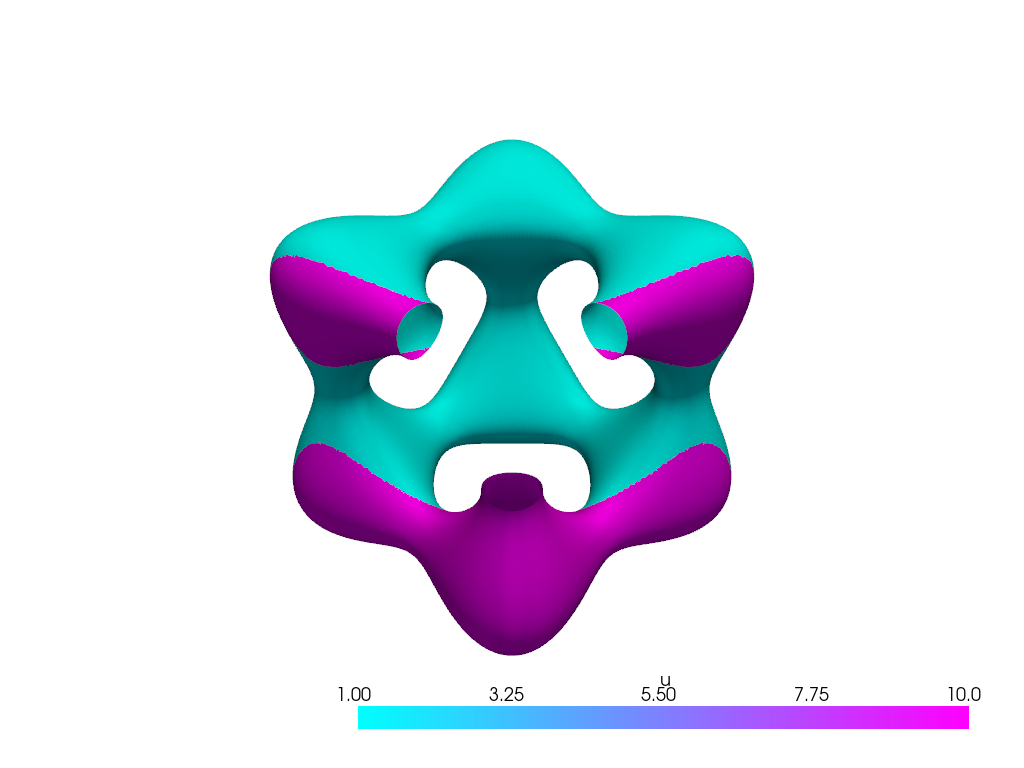}
\includegraphics[width=0.36\textwidth,trim={8cm 0 0 4cm},clip]{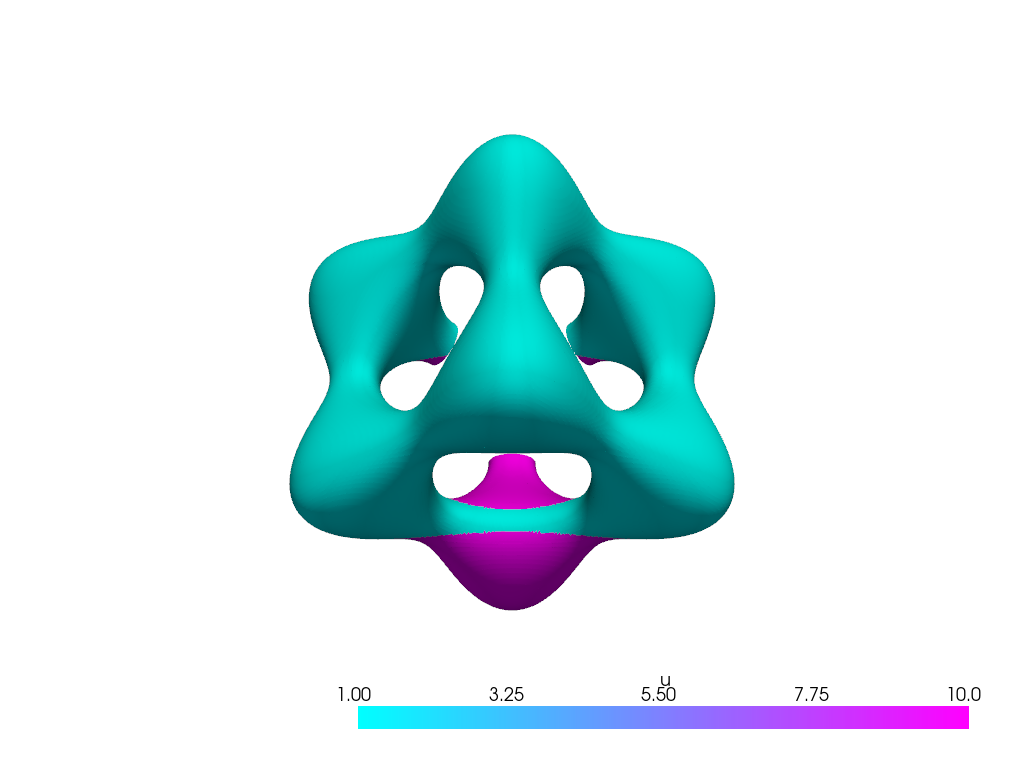} \\
\includegraphics[width=0.36\textwidth,trim={8cm 0 0 4cm},clip]{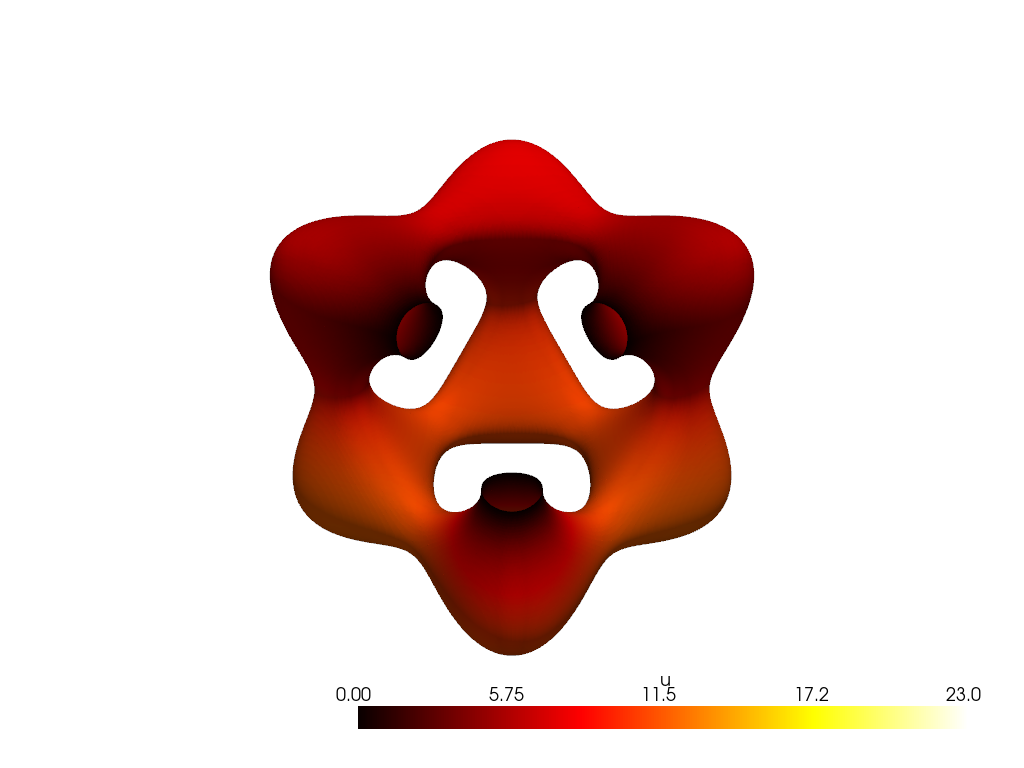} 
\includegraphics[width=0.36\textwidth,trim={8cm 0 0 4cm},clip]{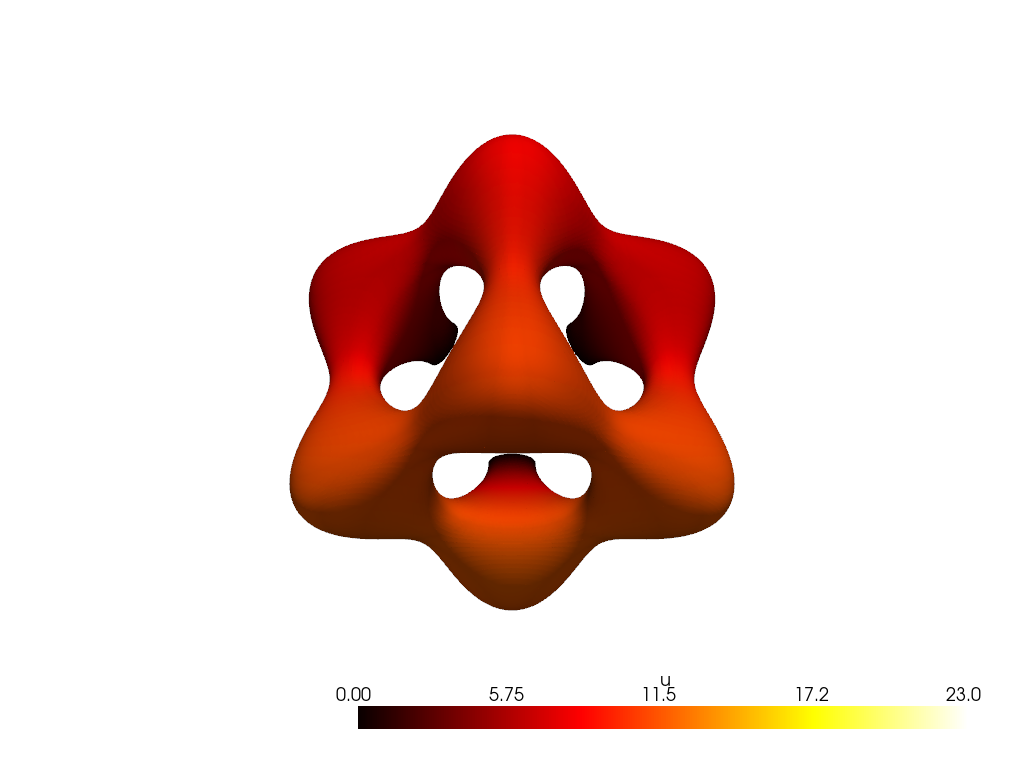} \\
\bigskip 
\includegraphics[width=0.36\textwidth,trim={8cm 0 0 4cm},clip]{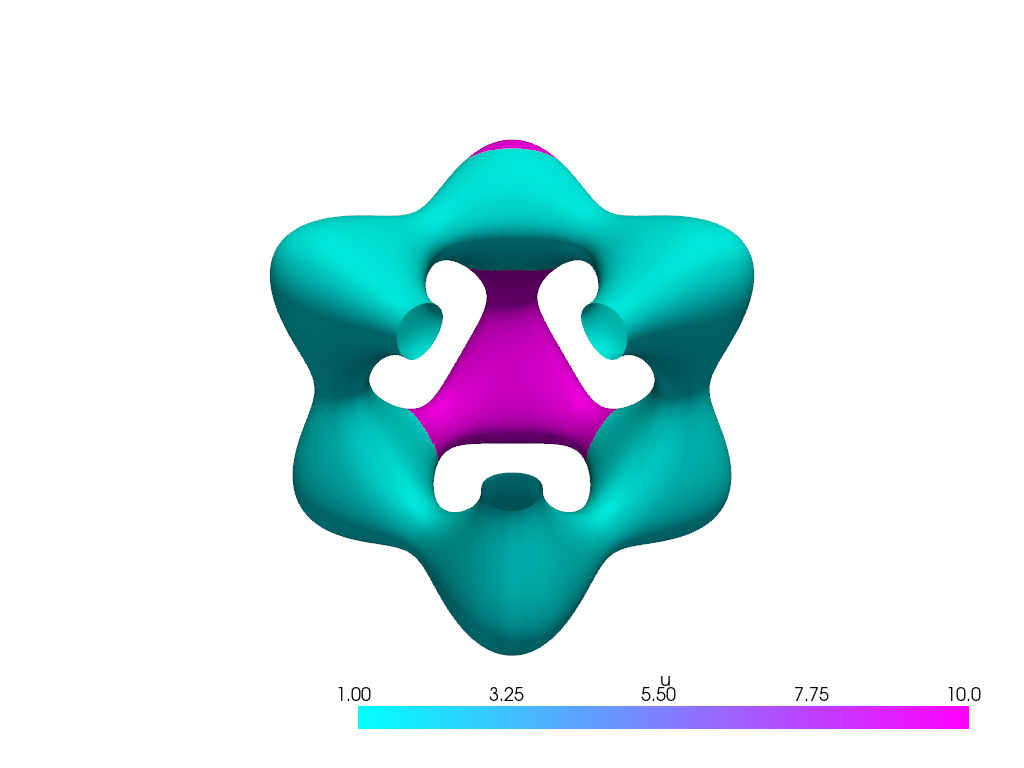}
\includegraphics[width=0.36\textwidth,trim={8cm 0 0 4cm},clip]{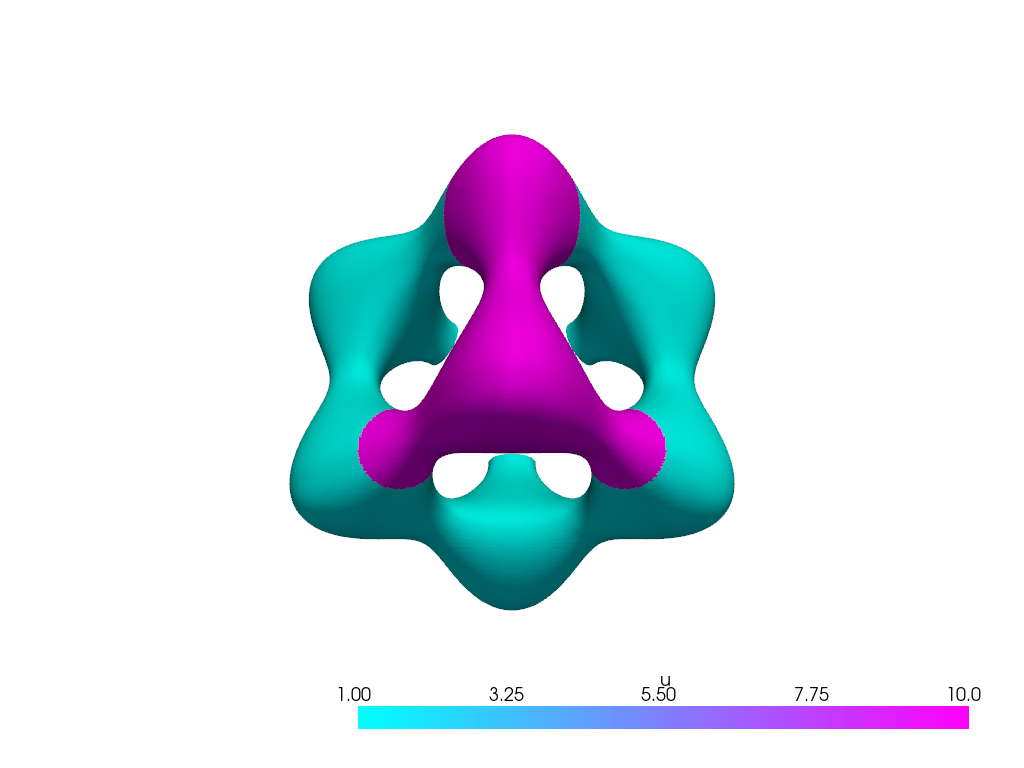} \\
\includegraphics[width=0.36\textwidth,trim={8cm 0 0 4cm},clip]{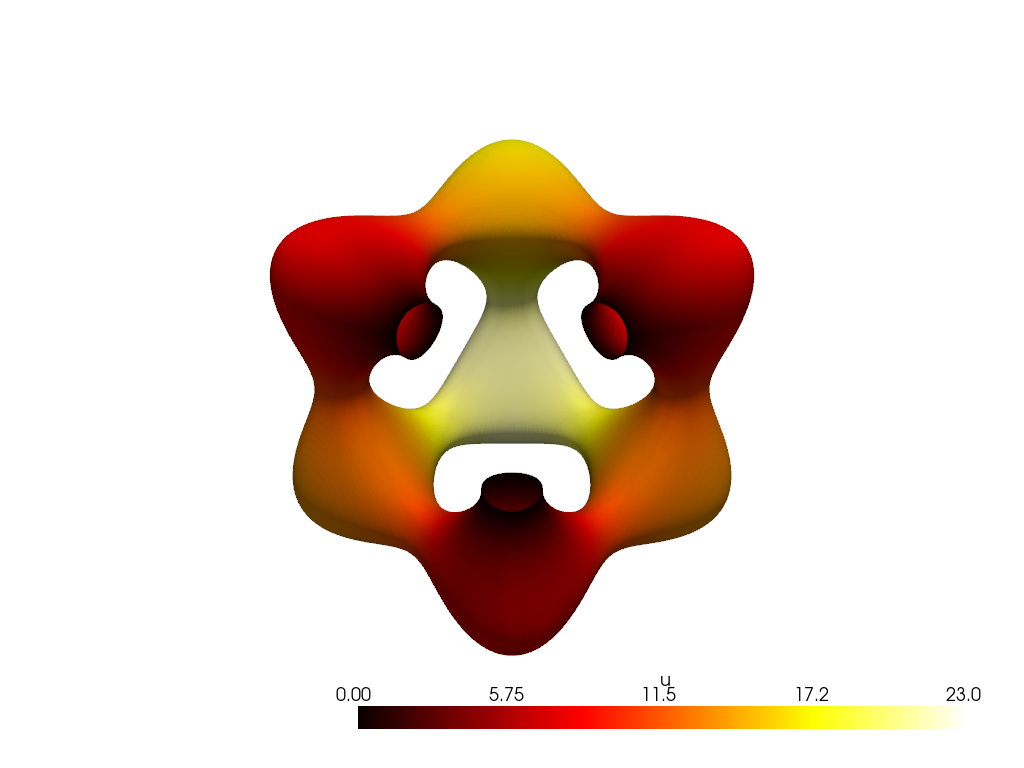}
\includegraphics[width=0.36\textwidth,trim={8cm 0 0 4cm},clip]{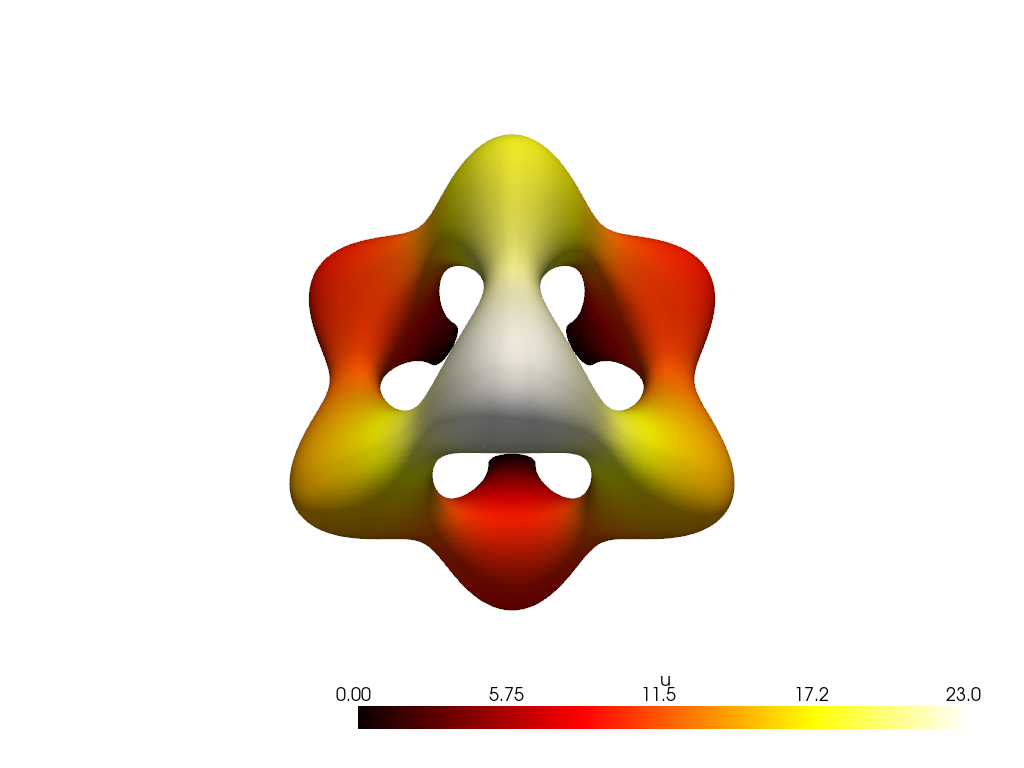}
\end{center}
\caption{{\bf (rows 1 and 2)} The initial $f_0$ and $u_0$ from the same two views as in \cref{fig:CompResGourat}(top). 
{\bf (rows 3 and 4)} The final $f_*$ and $u_*$ from the same two views. 
See \cref{s:CompResGourat} for details.}
\label{fig:CompResGourat2}
\end{figure}

\end{document}